\documentclass[a4paper]{amsart}
\usepackage{url}
\usepackage{amsmath}
\usepackage[makeroom]{cancel}
\usepackage{graphicx}
\usepackage[all]{xy}
\usepackage[mathscr]{eucal}
\usepackage{amsmath,amssymb,amsfonts}

\newtheorem{theorem}{Theorem}[section]
\newtheorem{lemma}[theorem]{Lemma}

\newtheorem{example}[theorem]{Example}

\newtheorem{proposition}[theorem]{Proposition}
\newtheorem{remark}[theorem]{Remark}

\usepackage{stackengine}
\usepackage{xcolor}
\usepackage[all]{xy}

\title{The structure of \'etale Boolean right restriction monoids}

\author{Mark V. Lawson}
\address{Mark V. Lawson, Department of Mathematics,
Maxwell Institute for Mathematical Sciences,
Heriot-Watt University,
Riccarton,
Edinburgh EH14 4AS,
UNITED KINGDOM}
\email{m.v.lawson@hw.ac.uk}

\begin{document}

%%%%%%%%%%%%%%%%%%%%%%%%%%%%%%%%%%%%%%%%%%%%%%%%%%%%%%%%%%%%%%%%%%%%%%%%%%%%%%%%%%%%%%%%%%%%%%
\begin{abstract}
In this paper, we describe \'etale Boolean right restriction monoids in terms of Boolean inverse monoids. 
We show that the Thompson groups arise naturally in this context.
\end{abstract}
\maketitle

%%%%%%%%%%%%%%%%%%%%%%%%%%%%%%%%%%%%%%%%%%%%%%%%%%%%%%%%%%%%%%%%%%%%%%%%%%%%%%%%%%%%%%%%%%%%%%%%%%
\section{Introduction}

The goal of this paper is to describe the structure of \'etale Boolean right restriction monoids in terms of Boolean inverse monoids motivated by \cite{CG, G1, G2}.
Etale Boolean right restriction monoids are interesting
because these are precisely the Boolean right restriction monoids 
whose associated categories, under non-commutative Stone duality,
are in fact groupoids  \cite[Theorem 5.2]{G2}.

Our starting point is the definition of the set of partial units\footnote{We prefer this term to that of `partial isomorphism' used by Cockett and Garner \cite{CG}.} 
of a Boolean right restriction monoid $S$. 
They form a Boolean inverse monoid $\mathsf{Inv}(S)$ \cite{Wehrung};
this was proved in \cite{CG, G2}, although we also give a full proof of this result in Lemma~\ref{lem:sleet}. 
A Boolean right restriction monoid $S$ is said to be \'etale if every element is a join of a finite number of partial units.
Etale Boolean right restriction monoids were first defined in \cite{G2}
though a special class of such monoids was actually used in \cite{Lawson2021b}.
A more general notion of `\'etale' was defined in \cite{CG}.

From the definition, we can see that there is a close connection
between \'etale Boolean right restriction monoids and Boolean inverse monoids.
This is made precise in Section 5.
In Theorem~\ref{them:one}, 
we show that the Boolean inverse monoid of partial units of an \'etale Boolean right restriction monoid determines the structure of that monoid.
Our main theorem,  Theorem~\ref{them:etale-main}, 
shows how to manufacture an \'etale Boolean right restriction monoid $T$ from a Boolean inverse monoid $S$
using tools developed in Section 4.
We call $T$ constructed in this way, the `companion' of $S$.
Section 6 provides some concrete examples of the theory we have developed including a discussion of the classical Thompson groups $G_{n,1}$.

%%%%%%%%%%%%%%%%%%%%%%%%%%%%%%%%%%%%%%%%%%%%%%%%%%%%%%%%%%%%%%%%%%%%%%%%%%%%%%%%%%%%%%%%%%%%%%%%%%%%%%%%%%%%%%%%%%%%%%%%%%%%%%%%%%%%
The rest of this introduction is given over to outlining some of the background needed to read this paper.

On every right restriction monoid is defined a partial order called the natural partial order,
which plays an important role in determining the structure of that monoid.
For this reason, we shall need some definitions and notation from the theory of posets.
Let $(X,\leq)$ be a poset.
If $Y \subseteq X$ define 
$$Y^{\downarrow} = \{x \in X \colon \exists y \in Y;  x \leq y\}.$$
In the case where $Y = \{y\}$, we write  $y^{\downarrow}$ 
instead of $\{y\}^{\downarrow}$. 
If $Y = Y^{\downarrow}$ we say that $Y$ is an {\em order ideal}.
If $P$ and $Q$ are posets, a function $\theta \colon P \rightarrow Q$ is said to be {\em order preserving}
if $p_{2} \leq p_{1}$ implies that $\theta (p_{2}) \leq \theta (p_{1})$.
An {\em order isomorphism} is a bijective order preserving map
whose inverse is also order preserving.

The {\em usual order} on the set of idempotents of any semigroup is defined by $e \leq f$ if $e = ef = fe$.

We shall use some basic topology in this paper \cite{Willard}.
Let $X$ be a set.
A set $\beta = \{ U_{i} \colon i \in I \}$
of subsets of $X$ is called a {\em basis}
if it satisfies two conditions:
the first is that $X = \bigcup_{U \in \beta} U$
and the second is that if $x \in U_{1} \cap U_{2}$, where $U_{1},U_{2} \in \beta$
then there exists $U \in \beta$ such that $x \in U \subseteq U_{1} \cap U_{2}$.
Bases are used to {\em generate} topologies on $X$.
A space $X$ is said to be {\em $0$-dimensional} if it has a basis consisting
of clopen sets.
A compact, Hausdorff, $0$-dimensional space is said to be {\em Boolean}.
It is important to distinguish {\em partial homeomorphisms} and  {\em local homeomorphisms}.
By a partial homeomorphism we mean a homeomorphism between two open subsets of a topological space.
A local homeomorphism is a continuous function $f \colon X \rightarrow Y$
such that for each $x \in X$ there is an open set $U_{x} \subseteq X$, containing $x$,
such that $f$ restricted to $U_{x}$ is a homeomorphism.

On a point of notation, in an inverse semigroup the idempotent $a^{-1}a$ is usually denoted by $\mathbf{d}(a)$
and $aa^{-1}$ is usually denoted by $\mathbf{r}(a)$.

If $X$ is a topological space then $\mathsf{I}(X)$ denotes the inverse monoid
of all partial homeomorphisms of $X$. In the case where $X$ is equipped with the discrete topology
we get back the symmetric inverse monoid on $X$.
If $X$ is a Boolean space then $\mathsf{I}^{\tiny cp}(X)$ denotes the Boolean inverse monoid
of all homeomorphisms between the clopen subsets of $X$.
This is the prototype for Boolean inverse monoids.\\

\noindent
{\bf Acknowledgements } 
I would like to thank many people for commenting on earlier versions of this paper or for general discussions on the topics of this paper,
amongst whom are:
Tristan Bice, 
Richard Garner,
Victoria Gould.
and 
Ganna Kudryavtseva.
I would also like to thank the anonymous referee for meticulously reading the submitted version of this paper
and for the many helpful comments provided.\\

%%%%%%%%%%%%%%%%%%%%%%%%%%%%%%%%%%%%%%%%%%%%%%%%%%%%%%%%%%%%%%%%%%%%%%%%%%%%%%%%%%%%%%%%%%%%%%%%%%%%%%%%%%%%%%%%%%%%%%%%%%%%%%%%%%%%%%%%%
\section{Right restriction semigroups}

Semigroups generalizing inverse semigroups have been studied by a number of authors at various times, work nicely summarized in \cite{Hollings}.
In addition, category theorists also became interested in categorical analogues of such semigroups,
motivated by a desire to axiomatize categories of partial functions,
notably in the work of Grandis \cite{Grandis} and Cockett (and his collaborators) \cite{Cockett}.
We shall focus on monoids in this paper.
The following well-known example is key and serves to motivate the class of semigroups we shall study in this paper.

\begin{example}\label{ex:first}
{\em Functions will (almost) always be computed from right to left.
The only exception is in Section~6.2, but this will be flagged up when it occurs.
Denote by  $\mathsf{PT}(X)$ the set of all partial functions defined on the (non-empty) set $X$.
See \cite{Gould}.
An element of  $\mathsf{PT}(X)$ has the form $f \colon A \rightarrow X$ where $A \subseteq X$.  
We call the subset $A$ the {\em domain of definition} of $f$;
this set will be denoted by $\mbox{\rm dom}(f)$. 
Denote by $f^{\ast}$ (called {\em $f$-star}) the identity function defined on $\mbox{\rm dom}(f)$. 
Observe that $ff^{\ast} = f$.
Whereas identity functions defined on subsets of $X$ are idempotents,
it is not true that all idempotents have this form.
The set of all those idempotents which are identities defined on subsets is denoted by $\mathsf{Proj}(\mathsf{PT}(X))$
and is called the set of {\em projections}.
Partial functions $f$ and $g$ can be compared using subset inclusion.
In fact, $f \subseteq g$ precisely when $f = gf^{\ast}$.
With respect to this order, the set of projections forms a Boolean algebra.
If $f,g \in \mathsf{PT}(X)$ then it is not true in general that $f \cup g \in \mathsf{PT}(X)$, since $f \cup g$ 
may not be a partial function.
It is a partial function precisely when $fg^{\ast} = gf^{\ast}$;
in this case, we say that $f$ and $g$ are {\em left-compatible}.
Observe also that $f^{\ast}g = g(fg)^{\ast}$ which expresses the fact that we are dealing with partial {\em functions}.
We may regard $\mathsf{PT}(X)$ as an algebra of type $(1,2)$ equipped with the star operation and the semigroup binary operation.
}
\end{example}
 
Example~\ref{ex:first} is a special case of the following definition.
We define a semigroup $S$ to be a {\em right restriction semigroup} if
it is equipped with a unary operation $a \mapsto a^{\ast}$ satisfying the following axioms:
\begin{itemize}
\item[{\rm (RR1).}] $(s^{\ast})^{\ast} = s^{\ast}$. 
\item[{\rm (RR2).}] $(s^{\ast}t^{\ast})^{\ast} = s^{\ast}t^{\ast}$.
\item[{\rm (RR3).}] $s^{\ast}t^{\ast} = t^{\ast}s^{\ast}$.
\item[{\rm (RR4).}] $ss^{\ast} = s$.
\item[{\rm (RR5).}] $(st)^{\ast} = (s^{\ast}t)^{\ast}$.
\item[{\rm (RR6).}] $t^{\ast}s = s(ts)^{\ast}$.
\end{itemize}
The unary operation $s \mapsto s^{\ast}$ is called {\em star}.
Denote by $\mathsf{Proj}(S)$ those elements $a$ such that $a^{\ast} = a$, called {\em projections}.
%If $X \subseteq S$ denote by $\mathsf{Proj}(X)$ the set $X \cap \mathsf{Proj}(S)$.

Let $S$ and $T$ be right restriction semigroups.
A {\em homomorphism} $\theta \colon S \rightarrow T$ of right restriction semigroups
is a semigroup homomorphism such that $\theta (a^{\ast}) = \theta (a)^{\ast}$.
Such homomorphisms map projections to projections.
%We may define {\em left restriction monoids}, dually, where we use $a^{+}$ rather than $a^{\ast}$.
%This paper is about right restriction monoids.

The lemma below is well-known but is included for context.
The proofs follow quickly from the axioms.
%Observe that the usual order on idempotents is given by $e \leq f$ precisely when $e = ef = fe$.

\begin{lemma}\label{lem:one} Let $S$ be a right restriction semigroup.
\begin{enumerate}
\item Each projection is an idempotent.
\item $ae = a$ implies that $a^{\ast} \leq e$ whenever $e$ is a projection.
\item If $S$ is a monoid then $1^{\ast} = 1$.
\item $(ab)^{\ast} \leq b^{\ast}$ for all elements $a,b \in S$.
\item If $S$ has a zero which is a projection, then $a = 0$ if and only if $a^{\ast} = 0$.
\item The product of projections is a projection.
\end{enumerate}
\end{lemma}

Let $S$ be a right restriction monoid.
The monoid of {\em total elements} of $S$, denoted by $\mathsf{Tot}(S)$,
is the monoid of all elements $a$ of $S$ such that $a^{\ast} = 1$.

\begin{remark}\label{rem:hopkins}{\em There is a Cayley-type representation theorem which says that
given any right restriction semigroup $S$ there is an embedding of right restriction semigroups
into $\mathsf{PT}(S)$.
Define $\phi \colon S \rightarrow \mathsf{PT}(S)$
where $\phi(a)$ is the partial function with domain of definition $a^{\ast}S$
such that $\phi(a)(x) = ax$.
This was first proved in \cite{Trokh}.\footnote{My thanks to Victoria Gould for supplying this reference.}
}
\end{remark} 

In a right restriction semigroup,
define a binary relation $a \leq b$ on $S$ by $a = ba^{\ast}$.
The following are useful.
Again, these results are well-known and are included for context.
The proofs are easy.

\begin{lemma}\label{lem:two}  Let $S$ be a right restriction semigroup.
\begin{enumerate}
\item If $a = be$, where $e$ is a projection, then $a \leq b$.
\item If $a = eb$, where $e$ is a projection, then $a \leq b$.
\item If $a \leq b$ then $a^{\ast} \leq b^{\ast}$.
\item The relation $\leq$ is a partial order. 
\item The semigroup $S$ is partially ordered with respect to $\leq$.
\item The set of projections forms an order ideal.
\end{enumerate}
\end{lemma}

We call $\leq$ the {\em natural partial order}.
This will be the only partial order we consider on a right restriction semigroup.
Observe that the natural partial order, when restricted to the projections, is the usual order on idempotents.
The following results are well-known and easy to prove.

\begin{lemma}\label{lem:spiked} Let $S$ be a right restriction semigroup. 
\begin{enumerate}

\item If $a,b \leq c$ and $a^{\ast} = b^{\ast}$ then $a = b$.

\item If $a,b \leq c$ then $ab^{\ast} = ba^{\ast}$.

\end{enumerate}
\end{lemma}

Part (2) of the above lemma motivates the following definition.
Define $a \sim_{l} b$, and say that $a$ and $b$ are {\em left-compatible},  if  $ab^{\ast} = ba^{\ast}$.
We say that projections $e$ and $f$ are {\em orthogonal}, written $e \perp f$,
if $ef = 0$.
We say that $a$ and $b$ are {\em left-orthogonal} if $a^{\ast} \perp b^{\ast}$.
Clearly, left-orthogonal elements are left-compatible.

\begin{remark}
{\em Homomorphisms of right restriction semigroups
preserve the natural partial order and left-compatibility.}
\end{remark}

The following is included for the sake of completeness.

\begin{lemma}\label{lem:yakult}
In an inverse semigroup, 
we have that $a \sim_{l} b$ if and only if $ab^{-1}$ is an idempotent.
\end{lemma}
\begin{proof} Suppose that  $a \sim_{l} b$.
Then $ab^{-1}b = ba^{-1}a$ by definition.
But then $ab^{-1}= ba^{-1}ab^{-1}$, which is an idempotent.
Conversely, suppose that $ab^{-1}$ is an idempotent.
Then $ab^{-1}b, ba^{-1}a \leq a,b$.
But $(ab^{-1}b)^{\ast} = (ba^{-1}a)^{\ast}$.
Thus $ab^{-1}b = ba^{-1}a$ by Lemma~\ref{lem:spiked}.
\end{proof}

In an inverse semigroup, we define {\em right-compatibility}
by $a \sim_{r} b $ if and only if $bb^{-1}a = aa^{-1}b$;
this is equivalent to $a^{-1}b$ being an idempotent by the dual
of Lemma~\ref{lem:yakult}.
In an inverse semigroup, we say that $a$ and $b$ are {\em compatible} if 
they are both left-compatible and right-compatible.
The proof of the following is immediate

\begin{lemma}\label{lem:yoghurt} In an inverse semigroup,
we have that $a \sim_{l} b$ if and only if $a^{-1} \sim_{r} b^{-1}$, 
and  $a \sim_{r} b$ if and only if $a^{-1} \sim_{l} b^{-1}$, 
\end{lemma}

The lemma below describes some of the important properties of the left-compatibity relation.

\begin{lemma}\label{lem:similarity} 
Let $S$ be a right restriction semigroup in which $a \sim_{l} b$.
\begin{enumerate}
\item If $a \sim_{l} b$ and $c \sim_{l} d$ then $ac \sim_{l} bd$. 
\item If $a \sim_{l} b$ and $x \leq a$ and $y \leq b$ then $x \sim_{l} y$. 
\end{enumerate}
\end{lemma}
\begin{proof} (1) We are given that $a \sim_{l} b$ and $c \sim_{l} d$.
This means that $ab^{\ast} = ba^{\ast}$ and $cd^{\ast} = dc^{\ast}$.
Thus $ab^{\ast}cd^{\ast} = ba^{\ast}dc^{\ast}$.
By axiom (RR6), we have that
$ac(bc)^{\ast}d^{\ast} = bd(ad)^{\ast}c^{\ast}$.
We focus on the left hand side of the above equality.
By axiom (RR2), we have that $ac(bc)^{\ast}d^{\ast} = ac [(bc)^{\ast}d^{\ast}]^{\ast}$.
By axiom (RR5), we have that $ac [(bc)^{\ast}d^{\ast}]^{\ast} = ac[bcd^{\ast}]^{\ast}$.
But $cd^{\ast} = dc^{\ast}$.
Thus $ac[bcd^{\ast}]^{\ast} =  ac[bdc^{\ast}]^{\ast}$.
Apply axiom (RR5), to get $ac[bdc^{\ast}]^{\ast} = ac[(bd)^{\ast}c^{\ast}]^{\ast}$.
By axiom (RR2), this is just  $ac(bd)^{\ast}c^{\ast}$.
We now use axiom (RR3) and axiom (RR4), to get $ac(bd)^{\ast}c^{\ast} = ac (bd)^{\ast}$.
We have proved that $ac(bc)^{\ast}d^{\ast} = ac (bd)^{\ast}$.
We may similarly show that $bd(ad)^{\ast}c^{\ast} = bd(ac)^{\ast}$.
The result now follows.

(2) We have that $ab^{\ast} = ba^{\ast}$ and $x = ax^{\ast}$ and $y = by^{\ast}$.
We have that 
$xy^{\ast} = ax^{\ast}b^{\ast}y^{\ast} = ba^{\ast}x^{\ast}y^{\ast}$
and
$yx^{\ast} = by^{\ast}a^{\ast}x^{\ast}$.
Since projections commute, we have shown that $x \sim_{l} y$.
\end{proof}

The following result tells us that being left-compatible is a property of the poset 
and the star operation alone and not the semigroup structure.

\begin{lemma}\label{lem:meets} Let $S$ be a right restriction semigroup. 
Then $a \sim_{l} b$ if and only if $a \wedge b$ exists and $(a \wedge b)^{\ast} = a^{\ast}b^{\ast}$.
\end{lemma}
\begin{proof} Suppose first $a \sim_{l} b$.
Put $x = ab^{\ast} = ba^{\ast}$.
Clearly, $x \leq a,b$ and $x^{\ast} = a^{\ast}b^{\ast}$.
Suppose that $z \leq a,b$.
Then $z^{\ast} \leq a^{\ast}, b^{\ast}$.
By the definition of the natural partial order,
we have that $z = az^{\ast} = bz^{\ast} \leq ab^{\ast} = x$.
We now prove the converse.
We have that $a \wedge b \leq a$ and so $a \wedge b = a(a \wedge b)^{\ast} = ab^{\ast}$.
By symmetry, we have that $a \wedge b = ba^{\ast}$.
It follows that $ab^{\ast} = ba^{\ast}$ and so $a$ and $b$ are left compatible.
\end{proof}

\begin{remark}{\em If $a \sim_{l} b$ then $a \wedge b = ab^{\ast}$ and so
this meet is {\em algebraically defined}.
It is therefore preserved under any homomorphism of right restriction semigroups.}
\end{remark}

Let $S$ be a right restriction semigroup.
An element $a \in S$ is said to be a {\em partial unit}
if there is an element $b \in S$ such that $ba = a^{\ast}$ and $ab = b^{\ast}$.
Clearly, every projection is a partial unit.
The set of all partial units of $S$ is denoted by $\mathsf{Inv}(S)$.

\begin{lemma}\label{lem:spilled-tea} Let $S$ be a right restriction semigroup
and let $a \in \mathsf{Inv}(S)$.
Suppose that $ax = x^{\ast}$ and $xa = a^{\ast}$,
and $ay = y^{\ast}$ and $ya = a^{\ast}$.
Then $x = y$.
\end{lemma}
\begin{proof} We have that $xa = ya$.
Thus $xay = yay$.
It follows that $y = xy^{\ast}$ and so $y \leq x$.
By symmetry, $x \leq y$ and so $x = y$.
\end{proof}

Let $S$ be a right restriction semigroup.
If $a \in \mathsf{Inv}(S)$ then we shall denote by $a^{-1}$ the unique
element guaranteed by Lemma~\ref{lem:spilled-tea} such that $aa^{-1} = (a^{-1})^{\ast}$ and $a^{-1}a = a^{\ast}$.
We shall now say more about the set $\mathsf{Inv}(S)$.
Most of the following was first proved as \cite[Lemma 2.14]{CG}.
We give proofs anyway.

\begin{lemma}\label{lem:holdo} Let $S$ be a right restriction semigroup.
\begin{enumerate}

\item If $a,b \in \mathsf{Inv}(S)$ then $ab \in \mathsf{Inv}(S)$. 

\item If $a,b \in \mathsf{Inv}(S)$ then $(ab)^{-1} = b^{-1}a^{-1}$.

\item If $a,b \in \mathsf{Inv}(S)$ then 
$a \leq b$ if and only if $a = aa^{-1}b$.

\item $\mathsf{Inv}(S)$ is an inverse semigroup with set of idempotents $\mathsf{Proj}(S)$.

\item $\mathsf{Inv}(S)$ is an order ideal.

\end{enumerate}
\end{lemma}
\begin{proof} (1) We have that
$aa^{-1} = (a^{-1})^{\ast}$ and $a^{-1}a = a^{\ast}$
and
$bb^{-1} = (b^{-1})^{\ast}$ and $b^{-1}b = b^{\ast}$.
Observe that
$ab(b^{-1}a^{-1}) =(b^{-1}a^{-1})^{\ast}$ and $(b^{-1}a^{-1})ab = (ab)^{\ast}$.

(2) Immediate from (1) above.

(3) Only one direction needs proving.
Suppose that $a \leq b$.
Then $a = ba^{\ast}$.
Now, $a$ has inverse $a^{-1}$ but $a^{\ast}b^{-1}$ is also an inverse.
So, by the uniqueness of inverses  guaranteed by Lemma~\ref{lem:spilled-tea},
we have that $a^{-1} = a^{\ast}b^{-1}$.
Thus $a^{-1} \leq b^{-1}$ and so
$a^{-1} = b^{-1}(a^{-1})^{\ast}$ from the definition of the natural partial order.
Taking inverses of both sides again, we get that $a = (a^{-1})^{\ast}b$ 
and so $a = aa^{-1}b$, as required.

(4) By (1) above, it follows that $\mathsf{Inv}(S)$ is closed under products. 
Suppose, now, that $a \in \mathsf{Inv}(S)$ is an idempotent.
Let $b \in S$ be such that $ab = b^{\ast}$, $ba = a^{\ast}$.
Since $a$ is an idempotent, we have that $a = (aba)(aba) = a(ba)(ab)a$.
But projections commute.
Thus $a = (a^{2}b)(ba^{2}) = (ab)(ba)$.
It follows that $a$ is the product of two projections and so is itself a projection.
Since the projections commute, we have proved that $\mathsf{Inv}(S)$ is an inverse semigroup.

(5) Suppose that $a \in \mathsf{Inv}(S)$ and that $b \leq a$.
We prove that $b \in \mathsf{Inv}(S)$. 
By definition, $b = ab^{\ast}$.
Observe that
$b^{\ast}a^{-1}b = a^{\ast}b^{\ast} = (ab^{\ast})^{\ast}$
and
$bb^{\ast}a^{-1} = ab^{\ast}a^{-1} = (a^{-1})(b^{\ast}a^{-1})^{\ast} = (b^{\ast}a^{-1})^{\ast}$.
This proves that $b \in \mathsf{Inv}(S)$.
\end{proof}

%%%%%%%%%%%%%%%%%%%%%%%%%%%%%%%%%%%%%%%%%%%%%%%%%%%%%%%%%%%%%%%%%%%%%%%%%%%%%%%%%%%%%%%%%%%%%%%%%%%%%%%%%%%%%%%%%%%%%%%%%%%%
\section{Order completeness properties of right restriction semigroups}

We shall study right restriction semigroups which satisfy
some order completeness properties with respect to the natural partial order.
A set of elements in a right restriction monoid is said to be {\em left-compatible}
if each pair of elements is left-compatible.
We say that a right restriction semigroup is {\em complete}
if every left-compatible set of elements has a join
and multiplication distributes over such joins {\em from the right}.
This means that if $\bigvee_{i} a_{i}$ exists then $\left( \bigvee_{i} a_{i} \right)b = \bigvee_{i} a_{i}b$;
this makes sense by Lemma~\ref{lem:similarity}.
A {\em homomorphism} between complete right restriction monoids is required to preserve
all left-compatible joins.
We say that a right restriction semigroup is {\em distributive}
if each pair of left-compatible elements has a join,
and multiplication distributes over such binary joins {\em from the right};
observe that in this case the projections form a distributive lattice (with a bottom but not necessarily a top).
A {\em homomorphism} between distributive right restriction monoids
is required to preserve all finite left-compatible joins.
A distributive right restriction semigroup is {\em Boolean}
if the set of projections actually forms a generalized Boolean algebra,
where a `generalized Boolean algebra' is a distributive lattice (again with a bottom but not necessarily a top)
in which each principal order ideal is a Boolean algebra.
Observe that the set of projections in a Boolean right retriction {\em monoid}
forms a Boolean algebra.
We can make similar definitions for inverse semigroups,
but require compatibility rather than left-compatibility.

In the following result, 
part (1) is proved in \cite[Proposition~2.14(i)]{CCG},
part (2) is a slightly expanded version of  \cite[Lemma 2.15]{KL},
part (3) is proved in \cite[Proposition 2.14(iii)]{CCG},
and
parts (4) and (5) are the analogues of parts (3) and (4) of \cite[Lemma 2.5]{Lawson2016} with almost identical proofs.

\begin{lemma}\label{lem:posy} Let $S$ be a right restriction semigroup.
\begin{enumerate}
\item If $\bigvee_{j \in I}a_{j}$ exists then $a_{i} = \left( \bigvee_{j \in I}a_{j} \right)a_{i}^{\ast}$ for each $i \in I$.

\item If both $\bigvee_{i \in I} a_{i}$ and $\bigvee_{i \in I} a_{i}^{\ast}$ exist
then
$\bigvee_{i \in I} a_{i}^{\ast}$ is a projection
and
$\left(  \bigvee_{i \in I} a_{i} \right)^{\ast} = \bigvee_{i \in I} a_{i}^{\ast}$.

\item Let $S$ be a complete right restriction semigroup
and if $I$ is finite then we may assume that $S$ is only a distributive right restriction semigroup.
Suppose that $\bigvee_{i \in I} a_{i}$ is defined.
Then $c\left( \bigvee_{i \in I} a_{i} \right) = \bigvee_{i \in I} ca_{i}$.

\item Let $S$ be a distributive right restriction semigroup.
Suppose that $\bigvee_{i=1}^{m} a_{i}$ and $c \wedge \left( \bigvee_{i=1}^{m} a_{i}   \right)$
both exist. Then all meets $c \wedge a_{i}$ exist, the join $\bigvee_{i=1}^{m} c \wedge a_{i}$ exists, 
and $c \wedge \left( \bigvee_{i=1}^{m} a_{i}   \right) = \bigvee_{i=1}^{m} c \wedge a_{i}$.

\item Let $S$ be a distributive right restriction semigroup.
Suppose that $b = \bigvee_{i=1}^{m} b_{i}$ exists, and all meets $a \wedge b_{i}$ exist.
Then the meet $a \wedge b$ exists, and is equal to $\bigvee_{i=1}^{m} a \wedge b_{i}$.
\end{enumerate}
\end{lemma}

\begin{remark}{\em
Part (3) of Lemma~\ref{lem:posy} 
tells us that although complete or distributive right restriction 
semigroups were defined in terms of multiplication distributing over any joins {\em from the right},
in fact, multiplication in such semigroups distributes over any joins also {\em from the left.}
}
\end{remark}

The following result is expected.
It was first proved in \cite{CG}.

\begin{lemma}\label{lem:sleet} In a Boolean right restriction monoid,
the set of partial units forms a Boolean inverse monoid.
\end{lemma}
\begin{proof} It is enough to prove that if $a$ and $b$ are partial units which are compatible then $a \vee b$ is a partial unit.
Since $a$ and $b$ are compatible, so too are $a^{-1}$ and $b^{-1}$
by Lemma~\ref{lem:yakult}.
It follows that the element
$a^{-1} \vee b^{-1}$ is defined.
We calculate $(a \vee b)(a^{-1} \vee b^{-1})$.
This equals $aa^{-1} \vee ab^{-1} \vee ba^{-1} \vee bb^{-1}$.
By assumption, both $ab^{-1}$ and $ba^{-1}$ are idempotents.
Observe that $ab^{-1} \leq aa^{-1}$ and $ba^{-1} \leq bb^{-1}$
Thus $(a \vee b)(a^{-1} \vee b^{-1}) = aa^{-1} \vee bb^{-1}$.
This is a join of projections and so is a projection by Lemma~\ref{lem:posy}.
Dually, we have that $(a^{-1} \vee b^{-1})(a \vee b)$ is a projection itself.
\end{proof}

We say that a Boolean right restriction monoid is {\em \'etale} if every element
is a join of a finite number of partial units.
We have now made all the definitions to state the following:\\

{\em The goal of this paper is to describe \'etale Boolean right restriction monoids in terms of their Boolean inverse monoids of partial units.}\\

The following will only be used in Remark~\ref{rem:functor}.

\begin{lemma}\label{lem:functor} Let $\theta \colon S \rightarrow T$ be a homomorphism between 
Boolean right restriction monoids.
Then $\theta$, restricted to $\mathsf{Inv}(S)$, is a homomorphism between Boolean inverse monoids.
\end{lemma}
\begin{proof}
It is immediate from the definition, that $\theta$ maps partial units to partial units.
Since $\theta$ preserves left-compatible joins, it preserves compatible joins.
\end{proof}

The join in a  right restriction monoid has the following important property.

\begin{lemma}\label{lem:left-orthogonal} 
Let $S$ be a Boolean right restriction monoid.
Suppose that $a = \bigvee_{i=1}^{m} a_{i}$ is a left-compatible join.
Then $a = \bigvee_{i=1}^{m} b_{i}$, a left-orthogonal join, where each $b_{i} \leq a_{i}$.
\end{lemma}
\begin{proof} We prove the case $n = 2$ directly.
Suppose that $x \sim_{l} y$.
We shall prove first that $x \vee y = x \vee y\overline{x^{\ast}}$.
This follows from the fact that $x^{\ast} \vee \overline{x^\ast} = 1$
and that $yx^{\ast} = xy^{\ast}$.
It is clear that $x$ and $y \overline{x^{\ast}}$ are left-orthogonal.
We now prove the general case by induction.
Let $a = \bigvee_{i=1}^{m+1} a_{i}$ be a left-compatible join.
We can write this as
$$a = \left( \bigvee_{i=1}^{m} a_{i} \right) \vee a_{m+1}.$$ 
By induction, this is
$$a = \left( \bigvee_{i=1}^{m} b_{i} \right) \vee a_{m+1}$$ 
where $b_{i} \leq a_{i}$ and the join in the brackets
is a left-orthogonal one.
Put $e = \bigvee_{i=1}^{m} b_{i}^{\ast}$
where we have used part (2) of Lemma~\ref{lem:posy}.
Then
$$a = \left( \bigvee_{i=1}^{m} b_{i} \right) \vee a_{m+1}\overline{e}$$ 
is a left-orthogonal join.
\end{proof}

%%%%%%%%%%%%%%%%%%%%%%%%%%%%%%%%%%%%%%%%%%%%%%%%%%%%%%%%%%%%%%%%%%%%%%%%%%%%%%%%%%%%%%%%%%%%%%%%%%%%%%%%%%%%%%%%%%%%%%%%%%%%%%%%%%%%
\section{Complete right restriction monoids}

The material in this section generalizes the notion of nucleus to be found in \cite[Chapter II, Section 2]{J} by way of what we did
in \cite[Section 4]{LL}.
{\em Throughout this section, $S$ will be a right restriction monoid.}

If $A \subseteq S$, define $A^{\ast} = \{a^{\ast} \colon a \in A\}$.
The proof of the following is straightforward; for the proof of part (3) use Lemma~\ref{lem:similarity}.

\begin{lemma}\label{lem:gervais} 
Let $S$ be a right restriction monoid.
\begin{enumerate}
\item If $A$ and $B$ are order ideals then $AB$ is an order ideal.
\item If $A$ is an order ideal then $A^{\ast}$ is an order ideal.
\item If $A$ and $B$ are both left-compatible sets then $AB$ is a left-compatible set.
\end{enumerate}
\end{lemma}

We shall generalize \cite{Schein} and prove that every right restriction monoid can be embedded in a complete right restriction monoid,
although the same construction can be found in \cite{CG}. 
We say that a subset of $S$ is {\em acceptable} if it is a left-compatible order ideal. 
Put $\mathsf{R}(S)$ equal to the set of all acceptable subsets of $S$.
Observe that subsets of $S$ of the form $a^{\downarrow}$ are acceptable by Lemma~\ref{lem:spiked}.
We may therefore define a function $\iota \colon S \rightarrow \mathsf{R}(S)$ by $\iota (a) = a^{\downarrow}$.

\begin{proposition}\label{prop:completion} Let $S$ be a right restriction monoid.
Then $\mathsf{R}(S)$ is a complete right restriction monoid
in which the natural partial order is subset inclusion and the projections are the order ideals of $\mathsf{Proj}(S)$.
In addition, the function  $\iota \colon S \rightarrow \mathsf{R}(S)$ is an embedding of right restriction monoids.
\end{proposition}  
\begin{proof} We first show that $\mathsf{R}(S)$ is a  right restriction monoid.
Using parts (1) and (3) of Lemma~\ref{lem:gervais}, the set $\mathsf{R}(S)$ is a semigroup under subset multiplication.
The set of all order ideals of $\mathsf{Proj}(S)$ is a set of idempotents (later, projections) for $\mathsf{R}(S)$.
Observe that the set of all projections, $1^{\downarrow}$, is an identity for $\mathsf{R}(S)$
which is therefore a monoid.
If we define a unary map on $\mathsf{R}(S)$ by $A \mapsto A^{\ast}$
then this is well-defined by part (2) of Lemma~\ref{lem:gervais}.
It remains to check that  $\mathsf{R}(S)$ is a right restriction monoid with respect to these operations.
The proofs of axioms (RR1), (RR2) and (RR3) are immediate.
To prove that axiom (RR4) holds, it is immediate that $A \subseteq AA^{\ast}$.
The proof that the reverse inclusion holds follows from the fact that $A$ is an order ideal.
The proof that axiom (RR5) holds is immediate.
It remains to show that axiom (RR6) holds.
It is immediate that $A^{\ast}B \subseteq B(AB)^{\ast}$ by axiom (RR6) which holds in $S$.
We now prove the reverse inclusion.
Let $b(ab_{1})^{\ast}$ be such that $b,b_{1} \in B$ and $a \in A$.
We shall prove that this is an element of $A^{\ast}B$.
Because $b,b_{1} \in B$, an acceptable set, 
we have that $bb_{1}^{\ast} = b_{1}b^{\ast}$.
We have that 
$$b(ab_{1})^{\ast} = b(ab_{1})^{\ast}b_{1}^{\ast} =  bb_{1}^{\ast}(ab_{1})^{\ast} = b_{1}b^{\ast}(ab_{1})^{\ast} =  b_{1}(ab_{1})^{\ast}b^{\ast}.$$
Thus
$$b(ab_{1})^{\ast} =  b_{1}(ab_{1})^{\ast}b^{\ast} = a^{\ast}b_{1}b^{\ast}.$$
But $b_{1}b^{\ast} \in B$ because $B$ is an order ideal.
We have therefore proved that $B(AB)^{\ast} \subseteq A^{\ast}B$.
This completes the proof that $\mathsf{R}(S)$ is a  right restriction monoid.

\begin{itemize}

\item Claim: if $A$ and $B$ are acceptable sets then $A \leq B$ in $\mathsf{R}(S)$ if and only if $A \subseteq B$.
We now prove the claim.
Suppose first that $A \leq B$.
By definition $A = BA^{\ast}$.
Let $a \in A$.
Then $a = bc^{\ast}$ where $b \in B$ and $c \in A$.
But $a \leq b$ and $b \in B$, an order ideal.
It follows that $a \in B$.
We have proved that $A \subseteq B$.
We now prove the converse.
Suppose that $A \subseteq B$.
We prove that $A = BA^{\ast}$.
Observe that $A \subseteq BA^{\ast}$.
Let $x \in BA^{\ast}$.
Then $x = bc^{\ast}$ where $c \in A$.
We have that $A \subseteq B$
and so $b \sim_{l} c$.
It follows that $bc^{\ast} = cb^{\ast}$.
Thus $x = cb^{\ast}$ and so $x \leq c$.
But $c \in A$ and $A$ is an order ideal and so $c \in A$, as required.

\item Claim: if $A$ and $B$ are acceptable sets then
$A \sim_{l} B$ in $\mathsf{R}(S)$ if and only if $A \cup B \in \mathsf{R}(S)$.
We now prove the claim.
Suppose first that $A \cup B \in \mathsf{R}(S)$.
We prove that $A \sim_{l} B$.
In fact, we shall prove that $AB^{\ast} \subseteq BA^{\ast}$ and then appeal to symmetry.
Let $ab^{\ast} \in AB^{\ast}$.
By assumption, $A \cup B$ is an acceptable set and so, in particular,
$ab^{\ast} = ba^{\ast}$.
It follows that $ab^{\ast} \in BA^{\ast}$. 
Suppose now that  $A \sim_{l} B$.
We shall prove that $A \cup B \in \mathsf{R}(S)$. 
It is enough to prove that if $a \in A$ and $b \in B$ then $a \sim_{l} b$.
We are given that $AB^{\ast} = BA^{\ast}$.
We have that $ab^{\ast} \in AB^{\ast}$ and so $ab^{\ast} = b_{1}a_{1}^{\ast}$ where $b_{1} \in B$ and $a_{1} \in A$.
By assumption, $a \sim_{l} a_{1}$ and $b \sim_{l} b_{1}$.
We claim that $ab^{\ast} \leq ba^{\ast}$ and symmetry delivers the result.
To prove the claim, we use the fact that  $ab^{\ast} = b_{1}a_{1}^{\ast}$.
Thus 
$$ab^{\ast} = b_{1}a_{1}^{\ast}b^{\ast}a^{\ast} = b_{1}b^{\ast}a_{1}^{\ast}a^{\ast} = bb_{1}^{\ast}(a_{1}^{\ast}a^{\ast}).$$

\end{itemize}

With these two results, we can now prove that $\mathsf{R}(S)$ is a {\em complete} right restriction monoid.
Let $\{A_{i} \colon i \in I\}$ be a left-compatible subset of $\mathsf{R}(S)$.
We claim that $A = \bigcup_{i \in I} A_{i} \in \mathsf{R}(S)$.
We therefore have to prove that $A$ is acceptable.
It is clearly an order ideal and so we have to show that any two elements of $A$ are left-compatible.
Without loss of generality, suppose that $a \in A_{i}$ and $b \in A_{j}$.
Then, by the above, $A_{i} \cup A_{j}$ is acceptable and so $a$ and $b$ are left-compatible.
We have proved that $A$ is an acceptable set.
Also, by what we proved above, we have that $A_{i} \leq A$ for any $i \in I$.
Suppose that $A_{i} \leq B$ for any $i \in I$ where $B$ is acceptable.
Then $A_{i} \subseteq B$ for any $i \in I$, by what we proved above.
Thus $A \subseteq B$ and so $A \leq B$.
We have therefore proved that $\mathsf{R}(S)$ has joins of left-compatible subsets.
Now, let $\{A_{i} \colon i \in I\}$ be a left-compatible subset of $\mathsf{R}(S)$ and let $B$ be any acceptable set.
We have to prove that 
$\left( \bigcup_{i \in I}  A_{i} \right)B  =   \bigcup_{i \in I}  A_{i}B$.
However, this is true on set-theoretic grounds alone.
This completes the proof that $\mathsf{R}(S)$ is a  complete right restriction monoid.

It remains to prove that the function $\iota$ is a homomorphism of right restriction monoids.
It is immediate that it is an embedding.
Observe that $a^{\downarrow}b^{\downarrow} = (ab)^{\downarrow}$;
this is true since if $a' \leq a$ and $b' \leq b$ then $a'b' \leq ab$
and if $x \leq ab$ then $x = a(bx^{\ast}) = a'b'$
where $a' = a$ and $b'= bx^{\ast} \leq b$. 
We therefore have a homomorphism of semigroups
which is also a monoid homomorphism since the identity of $\mathsf{R}(S)$ is $1^{\downarrow}$.
Thus we finish if we show that $\iota$ is a homomorphism of right restriction monoids.
This requires us to show that $\iota (a^{\ast}) = \iota (a)^{\ast}$.
Let $x \in \iota (a^{\ast})$.
Then $x \leq a^{\ast}$ and so is a projection.
Consider the element $ax$.
Because $x$ is a projection, we have that $ax \leq a$ and so $ax \in \iota (a)$.
But $(ax)^{\ast} = a^{\ast}x = x$.
It follows that $x \in \iota (a)^{\ast}$.
On the other hand, let $b^{\ast} \in \iota (a)^{\ast}$ where $b \leq a$.
It follows that $b^{\ast} \leq a^{\ast}$ and so $b^{\ast} \in \iota (a^{\ast})$. 
\end{proof}

The above procedure can be applied, {\em inter alia}, when $S$ is a Boolean inverse monoid.
This is the only case that will interest us in Section~5.

We can say more about the function $\iota \colon S \rightarrow \mathsf{R}(S)$
and so the construction of $\mathsf{R}(S)$. 

\begin{proposition}\label{prop:truss}
The map $\iota \colon S \rightarrow \mathsf{R}(S)$ is universal for right restriction monoid homomorphisms
to complete right restriction monoids.
\end{proposition}
\begin{proof} Let $T$ be a complete right restriction monoid and let $\alpha \colon S \rightarrow T$
be a monoid homorphism of right restriction monoids.
Define $\beta \colon \mathsf{R}(S) \rightarrow T$ by
$\beta (A) = \bigvee_{a \in A} \alpha (a)$.
This makes sense since the elements of $A$ are pairwise left-compatible
and left-compatibility is preserved by homomorphisms of right restriction semigroups.
We now calculate $\beta (\iota (a))$.
By definition this is  $\beta (a^{\downarrow})$ which is
$\bigvee_{x \leq a} \alpha (x)$.
Observe that $x \leq a$ implies that $\alpha (x) \leq \alpha(a)$.
It follows that  $\beta (\iota (a)) = \alpha (a)$.
We show that $\beta$ is a right restriction monoid homomorphism.
It is immediate from the definitions that this is a monoid homomorphism.
We need to prove that it is a homomorphism of right restriction semigroups.
Let $A$ be an acceptable set.
Then, by definition,
$$\beta (A^{\ast}) = \bigvee_{a^{\ast} \in A^{\ast}} \alpha (a^{\ast}).$$
But $\alpha$ is a homomorphism of right restriction semigroups.
Thus $\alpha (a^{\ast}) = \alpha (a)^{\ast}$.
Now apply part (2) of Lemma~\ref{lem:posy} to get
$$\beta (A^{\ast}) = \left( \bigvee_{a^{\ast} \in A^{\ast}} \alpha (a) \right)^{\ast}.$$
We now use the fact that $a \in A$ if and only if $a^{\ast} \in A^{\ast}$.
This gives us 
$$\beta (A^{\ast}) = \left( \bigvee_{a \in A} \alpha (a) \right)^{\ast} = \beta (A)^{\ast}.$$
We have therefore shown that $\beta$ is a homomorphism of right restriction semigroups.
We show that $\beta$ preserves arbitrary left-compatible joins.
Let $\{A_{i} \colon i \in I \}$ be a left-compatible set in $\mathsf{R}(S)$.
Put $A = \bigcup_{i \in I} A_{i}$, the join of the $A_{i}$ in  $\mathsf{R}(S)$.
By definition
$$\beta (A) = \bigvee_{a \in A} \alpha (a) = \bigvee_{a \in A_{i}, i \in I} \alpha (a) = \bigvee_{i \in I}\left( \bigvee_{a \in A_{i}} \alpha (a)\right).$$
But this is equal to
$$\bigvee_{i \in I} \beta (A_{i}).$$
We finish off by proving the categorical property we need. 
Observe that for any acceptable set $A$ we have that $A = \bigcup_{a \in A} a^{\downarrow}$, because $A$ is an order ideal.
It now follows that if $\beta'$ is a homomorphism of complete right restriction monoids 
such that $\alpha = \beta' \iota$ then $\beta' = \beta$.
\end{proof}

Proposition~\ref{prop:completion} tells us how to manufacture complete right restriction monoids
from monoids that are merely right restriction monoids. 

%%%%%%%%%%%%%%%%%%%%%%%%%%%%%%%%%%%%%%%%%%%%%%%%%%%%%%%%%%%%%%%%%%%%%%%%%%%%%%%%%%%%%%%%%%%%%%%%%%%%%%%%%%%%%%%%%%%%%%%%%%%%%%%%%%%
For the rest of this section, we shall work with an {\em arbitrary} complete right restriction monoid $S$.
We shall now describe some homomorphisms between complete right restriction monoids 
in terms of a generalization of the notion of a nucleus discussed in \cite{J}.
Let $\theta \colon S \rightarrow T$ be a surjective homomorphism
between complete right restriction monoids.
We shall say that $\theta$ is {\em projection pure}
if $\theta (a) \sim_{l} \theta (b)$ implies that $a \sim_{l} b$.
The terminology we use generalizes that to be found in the theory of pseudogroups \cite{LL}
where the notion of nucleus was first generalized.  

\begin{lemma}\label{lem:projection-pure-projection} Let $\theta \colon S \rightarrow T$ be a projection pure homomomorphism
of complete right restriction monoids.
If $\theta (a)$ is a projection then $a$ is a projection
\end{lemma}
\begin{proof} Suppose that $\theta (a)$ is a projection.
Thus $\theta (a)^{\ast} = \theta (a)$.
Observe that $\theta (a) \sim_{l} \theta (a^{\ast})$.
Thus $a \sim_{l} a^{\ast}$, since $\theta$ is projection pure.
It follows from this that $a = a^{\ast}$. 
\end{proof}

Given a projection pure homomorphism 
$\theta \colon S \rightarrow T$ 
define
$\theta_{\ast} \colon T \rightarrow S$ 
by $\theta_{\ast}(t) = \bigvee_{\theta (s) \leq t} s$.
This makes sense because $\theta$ is projection pure.
This map has the following properties:
\begin{enumerate}
\item $\theta_{\ast}$ maps projections to projections.
\item $\theta_{\ast}$ is order preserving.
\item $\theta_{\ast}(t)\theta_{\ast}(t') \leq \theta_{\ast}(tt')$.
\item $s \leq \theta_{\ast} \theta (s)$.
\item $\theta \theta_{\ast}(t) \leq t$.
\end{enumerate}
We deduce from the above properties the following two equations
$$\theta = \theta \theta_{\ast} \theta \text{ and } \theta_{\ast} = \theta_{\ast} \theta \theta_{\ast}.$$
Define $\nu (s) = \theta_{\ast} \theta (s)$.
Then $\nu$ is a nucleus on $S$ in the following sense.
A function $\nu \colon S \rightarrow S$ is called a {\em nucleus} if it satisfies the following
six conditions:
\begin{itemize}
\item[{\rm (N1).}] $a \leq \nu (a)$.
\item[{\rm (N2).}] $a \leq b$ implies that $\nu (a) \leq \nu (b)$.
\item[{\rm (N3).}] $\nu^{2}(a) = \nu (a)$.
\item[{\rm (N4).}] $\nu (a) \nu (b) \leq \nu (ab)$.
\item[{\rm (N5).}] If $e$ is a projection then $\nu (e)$ is a projection.
\item[{\rm (N6).}] $\nu (a^{\ast}) = \nu (\nu (a)^{\ast})$.
\end{itemize}
The fact that (N6) holds follows from the fact that $\theta (a^{\ast}) = \theta (a)^{\ast}$.
We started with a projection pure homomorphism $\theta \colon S \rightarrow T$.
From this we constructed a nucleus $\nu$.
We shall show that $\nu$ contains all the information needed to reconstitute $T$
upto isomorphism. 

\begin{lemma}\label{lem:nucleus-properties} 
Let $\nu$ be a nuclus defined on the complete right restriction monoid $S$.
Then 
$$\nu (ab) = \nu (a \nu (b)) = \nu (\nu (a)b) = \nu (\nu (a) \nu (b)).$$
\end{lemma}
\begin{proof} By (N1), we have that $a \leq \nu (a)$ and $b \leq \nu (b)$.
In particular, $ab \leq \nu (a)\nu (b)$.
Thus by (N2), we have that $\nu (ab) \leq \nu (\nu (a)\nu (b))$.
But by (N4), we have that $\nu (a)\nu (b) \leq  \nu(ab)$.
Thus by (N2), we have that $\nu (\nu (a)\nu (b)) \leq \nu^{2}(ab)$.
But by (N3), we have that  $\nu (\nu (a)\nu (b)) \leq \nu (ab)$.
We have therefore proved that 
$\nu (ab) = \nu(\nu (a) \nu (b))$.
The other cases are proved similarly. 
\end{proof}

Let $S$ be a complete right restriction monoid equipped with a nucleus $\nu$.
Define
$$S_{\nu} = \{a \in S \colon \nu (a) = a\},$$
the set of {\em $\nu$-closed} elements.
Define $\cdot$ on $S_{\nu}$ by
$$a \cdot b = \nu (ab).$$ 
The following result tells us exactly how to build a new complete right restriction monoid
from an old one equipped with a nucleus.

\begin{proposition}\label{prop:as} Let $\nu$ be a nucleus defined on the complete right restriction monoid $S$.
Then $(S_{\nu},\cdot)$ is also a complete right restriction monoid.
\end{proposition}
\begin{proof} By Lemma~\ref{lem:nucleus-properties}, $(S_{\nu},\cdot)$ is a semigroup.
It is, in fact, a monoid with identity $\nu (1)$
since
$a \cdot \nu (1) = \nu (a \nu (1))$
where $a \in S_{\nu}$.
By Lemma~\ref{lem:nucleus-properties},
we have that $\nu (a \nu (1)) = \nu (a1) = \nu (a) = a$.
We have proved that $\nu(1)$ is a right identity.
It is a left identity by symmetry.

We now prove that $S_{\nu}$ is a right restriction monoid.
Put $\mathsf{Proj} (S_{\nu}) = \{\nu (e) \colon e \in \mathsf{Proj}(S)\}$.
This is a set of projections of $S$ by axiom (N5).
Thus $\nu (1)$ is a projection.
If $a \in S_{\nu}$, define 
$$a^{\circ} = \nu (a^{\ast}).$$
This will be our `star' operation.
This is a projection by (N5).
We now show that the axioms for a right restriction semigroup hold.
Let $a,b \in S_{\nu}$.

\begin{itemize}

\item (RR1) holds:
$(a^{\circ})^{\circ} = \nu (\nu (a^{\ast})^{\ast}) = \nu (\nu (a^{\ast})) = \nu (a^{\ast}) = a^{\circ}$
by (N5) and (N3).

\item (RR2) holds: $(a^{\circ} \cdot b^{\circ})^{\circ} = (\nu(a^{\circ}b^{\circ}))^{\circ}$.
This is equal to $\nu (\nu(a^{\ast})\nu(b^{\ast}))^{\circ} = \nu (a^{\ast}b^{\ast})^{\circ}$ using Lemma~\ref{lem:nucleus-properties}.
This is equal to $\nu (\nu (a^{\ast}b^{\ast})^{\ast}) = \nu (a^{\ast}b^{\ast})$ using (N6).
Whereas, $a^{\circ} \cdot b^{\circ} = \nu (a^{\circ}b^{\circ}) = \nu (\nu (a^{\ast}) \nu (b^{\ast})) = \nu (a^{\ast}b^{\ast})$
using  Lemma~\ref{lem:nucleus-properties}.

\item (RR3) holds: $a^{\circ} \cdot b^{\circ} = \nu (a^{\circ}b^{\circ}) = \nu (b^{\circ}a^{\circ}) = b^{\circ} \cdot a^{\circ}$
where we have used (N5).

\item (RR4) holds: $a \cdot a^{\circ} = \nu (aa^{\circ}) = \nu (a\nu(a^{\ast})) = \nu (aa^{\ast}) = \nu (a) = a$ by Lemma~\ref{lem:nucleus-properties}.

\item (RR5) holds: $(a \cdot b)^{\circ} = \nu(\nu (ab)^{\ast}) = \nu ((ab)^{\ast})$ by (N6).
On the other hand, $(a^{\circ} \cdot b)^{\circ} = \nu ( \nu (\nu(a^{\ast})b)^{\ast}) = \nu (\nu (a^{\ast}b)^{\ast}) = \nu ((a^{\ast}b)^{\ast}) = \nu ((ab)^{\ast})$
by (N3) and (N6).

\item (RR6) holds:
$b^{\circ} \cdot a = \nu (b^{\circ}a) = \nu (\nu(b^{\ast})a) = \nu (b^{\ast}a)$ by Lemma~\ref{lem:nucleus-properties}.
Whereas
$a \cdot (b \cdot a)^{\circ} = \nu (a(b \cdot a)^{\circ}) = \nu (a \nu(ba)^{\circ}) = \nu (a \nu (\nu (ba)^{\ast})) = \nu (a(ba)^{\ast})$
by (N6) and  Lemma~\ref{lem:nucleus-properties}.
\end{itemize}

Thus, we have proved that $(S_{\nu},\cdot)$ is a right restriction monoid.
We now establish a number of claims.
\begin{itemize}

\item Let $a,b \in S_{\nu}$.
Denote the natural partial order on $S_{\nu}$ by $\preceq$.
Claim: $a \preceq b$ if and only if $a \leq b$.
Proof of claim.
Suppose first that $a \preceq b$.
Then $a = b \cdot a^{\circ}$.
Thus $a = \nu (ba^{\circ}) = \nu (b \nu(a^{\ast})) = \nu (ba^{\ast})$
using Lemma~\ref{lem:nucleus-properties}.
But $ba^{\ast} \leq b$ and so $\nu (ba^{\ast}) \leq \nu (b) = b$ by (N2).
Thus $a \leq b$.
Suppose now that $a \leq b$.
This means that $a = ba^{\ast}$.
Thus $a = \nu(ba^{\ast}) = \nu (b \nu (a^{\ast})) = b \cdot a^{\circ}$
by Lemma~\ref{lem:nucleus-properties}.
Whence we have proved that $a \preceq b$.

\item Claim:  $a \sim_{l} b$ in $S_{\nu}$ if and only if $a \sim_{l} b$ in $S$.
Proof of claim.
Suppose, first, that $a \sim_{l} b$ in $S_{\nu}$.
Then $a \cdot b^{\circ} = b \cdot a^{\circ}$.
This means that $\nu (a \nu (b^\ast)) = \nu (b \nu (a^{\ast}))$.
Consequently, we have that $\nu (ab^{\ast}) = \nu (ba^{\ast})$
by  Lemma~\ref{lem:nucleus-properties}.
But $ab^{\ast} \leq \nu (ab^{\ast})$ by (N1).
Similarly, $ba^{\ast} \leq  \nu (ba^{\ast})$.
It follows that $ab^{\ast} \sim_{l} ba^{\ast}$ by Lemma~\ref{lem:spiked}.
Thus $ab^{\ast}(ba^{\ast})^{\ast} = ba^{\ast}(ab^{\ast})^{\ast}$.
Whence $ab^{\ast} = ba^{\ast}$ and so $a \sim_{l} b$ in $S$.
Now, suppose that $a \sim_{l} b$ in $S$.
This means that $ab^{\ast} = ba^{\ast}$.
But $a \cdot \nu (b^{\ast}) = a \cdot b^{\circ}$
and $a \cdot \nu (b^{\ast}) = \nu (a \nu (b^{\ast})) = \nu (ab^{\ast})$ by Lemma~\ref{lem:nucleus-properties}.
It follows that $a \sim_{l} b$ in $S$.

\item Claim: if $X = \{a_{i} \colon i \in I\}$ is a left-compatible set in $S_{\nu}$
then the join in $S_{\nu}$ of $X$ exists, it is denoted by  $\bigsqcup_{i \in I} a_{i}$,
and is equal to $\nu\left(  \bigvee_{i \in I} a_{i} \right)$.
Proof of claim.
By the above, this is a left-compatible set in $S$.
It therefore has a join $a$ in $S$.
We claim that $\nu (a)$ is the join of $X$ in $S_{\nu}$.
It is an element of $S_{\nu}$ by (N2).
We have that $a_{i} \leq a$ for all $i$.
Thus  $a_{i} \leq \nu (a)$ for all $i$ by (N2).
Let $a_{i} \leq b$ for all $i$ where $b \in S_{\nu}$.
This means that $a_{i} \leq b$ in $S$.
Thus $a \leq b$.
It follows that $\nu (a) \leq b$ in $S_{\nu}$.
Thus the join of $X$ exists in $S_{\nu}$.
It follows that all joins of left-compatible subsets of $S_{\nu}$ exist.

\item Claim: $\nu \left( \bigvee_{i \in I} \nu (a_{i})  \right) = \nu (\bigvee_{i \in I} a_{i})$,
where the $a_{i}$ are arbitrary elements of $S$ which form a left-compatible set.
Proof of claim.
We have that $a_{i} \leq \bigvee_{i \in I} a_{i}$.
Thus  $\nu (a_{i}) \leq \nu \left( \bigvee_{i \in I} a_{i} \right)$ by (N2).
Thus  $\bigvee_{i \in I} \nu (a_{i}) \leq \nu \left( \bigvee_{i \in I} a_{i} \right)$.
Whence  $\nu \left(  \bigvee_{i \in I} \nu (a_{i}) \right) \leq \nu \left( \bigvee_{i \in I} a_{i} \right)$
using (N2).
To prove the reverse inequality,
we start with $a_{i} \leq \nu (a_{i})$ by (N1).
It follows that $\bigvee_{i \in I} a_{i} \leq \bigvee_{i \in I} \nu (a_{i})$ and so
$\nu \left( \bigvee_{i \in I} a_{i} \right) \leq \nu \left( \bigvee_{i \in I} \nu (a_{i}) \right)$.

\item Claim: $\left( \bigsqcup_{i \in I} a_{i} \right) \cdot b = \bigsqcup_{i \in I} a_{i} \cdot b$
in $S_{\nu}$.
Proof of claim.
We have that 
$\left( \bigsqcup_{i \in I} a_{i} \right) \cdot b = \nu \left( \bigvee_{i \in I} a_{i}b \right)$
and
$\bigsqcup_{i \in I} a_{i} \cdot b = \nu \left( \bigvee_{i \in I} \nu (a_{i}b)\right)$.
\end{itemize}

The result now follows by what we proved above.
\end{proof}

We can now show that nuclei give rise to all  projection pure homomorphisms.

\begin{theorem}
Let $\theta \colon S \rightarrow T$ be a surjective projection pure homomorphism.
Define the nucleus $\nu (s) = \theta_{\ast} \theta (s)$.
Then $T$ is isomorphic to $S_{\nu}$ as a complete right restriction monoid.
\end{theorem}
\begin{proof} Define $\alpha \colon S_{\nu} \rightarrow T$ by $\alpha (s) = \theta (s)$.
Then $\alpha (s_{1} \cdot s_{2}) = \theta (\nu (s_{1}s_{2})) = \theta (s_{1}s_{2}) = \alpha (s_{1}) \alpha(s_{2})$.
Thus $\alpha$ is a semigroup homomorphism.
Suppose that $\alpha (s_{1}) = \alpha (s_{2})$.
Then $\theta (s_{1}) = \theta (s_{2})$.
It follows that $\nu (s_{1}) = \nu (s_{2})$.
Thus $s_{1} = s_{2}$, since both elements are assumed closed.
We have shown that $\alpha$ is an injective homomorphism.
Let $t \in T$.
Then since $\theta$ is assumed surjective, there is $s \in S$ such that $\theta (s) = t$.
But $\theta = \theta \nu$.
Thus $\theta (\nu (s)) = t$.
But $\nu (s)$ is closed.
We have shown that $\alpha$ is an isomorphism of semigroups.
Let $s$ be closed.
Then 
$$\alpha (s^{\circ}) = \alpha (\nu (s^{\ast})) = \theta (\nu (s^{\ast})) = \theta (s^{\ast}) = \theta (s)^{\ast} = \alpha (s)^{\ast}.$$
Thus $\alpha$ is an isomorphism of right restriction monoids.
We now refer the reader to Proposition~\ref{prop:as}.
Finally, suppose that $a_{i}$, where $i \in I$, is a set of left-compatible elements in $S_{\nu}$.
We calculate $\alpha (\bigsqcup_{i \in I} a_{i})$. 
But $\bigsqcup_{i \in I} a_{i} = \nu (\bigvee_{i \in I} a_{i})$.
It now follows that $\alpha$ preserves any left-compatible joins.
Thus $\alpha$ is an isomorphism of complete right restriction monoids.
\end{proof}

%%%%%%%%%%%%%%%%%%%%%%%%%%%%%%%%%%%%%%%%%%%%%%%%%%%%%%%%%%%%%%%%%%%%%%%%%%%%%%%%%%%%%%%%%%%%%%%%%%%%%%%%%%%%%%%%%%%%%%%%%%%
We are actually interested in constructing Boolean right restriction monoids.
The following concept is just what we need to cut down from arbitrary joins to finitary ones.
Let $S$ be a complete right restriction monoid.
An element $a \in S$ is said to be {\em finite} if
whenever $a \leq \bigvee_{i \in I} a_{i}$ then $a \leq \bigvee_{i=1}^{m}a_{i}$,
relabelling if necessary.\footnote{It might be better if we had said `compact' rather than `finite'. This technical meaning of the word `finite' is what we use below.}
Denote the set of finite elements of a complete right restriction monoid $S$ by $\mathsf{fin}(S)$.

\begin{lemma}\label{lem:klavin} Let $S$ be a complete right restriction monoid.
\begin{enumerate}
\item $a$ is finite if and only if $a^{\ast}$ is finite.
\item If $a$ and $b$ are finite and $a \sim_{l} b$ then $a \vee b$ is finite.
\item If the finite elements are closed under multiplication then they form a distributive right restriction semigroup.
\end{enumerate}
\end{lemma}
\begin{proof} (1) Suppose that $a$ is finite.
We prove that $a^{\ast}$ is finite.
Suppose that $a^{\ast} \leq \bigvee_{i \in I} b_{i}$.
Then from the definition of tha natural partial order, we have that
$a^{\ast} = \left( \bigvee_{i \in I} b_{i} \right)a^{\ast}$.
Thus $a^{\ast} = \bigvee_{i \in I} b_{i}a^{\ast}$.
By part (6) of Lemma~\ref{lem:two}, each $b_{i}a^{\ast}$ is a projection.
From axiom (RR4), we have that 
$a = \bigvee_{i \in I} ab_{i}a^{\ast}$.
But $a$ is finite.
Thus, relabelling if necessary,
$a = \bigvee_{i = 1}^{n} ab_{i}a^{\ast}$.
Now apply part (2) of Lemma~\ref{lem:posy}, to deduce that
$a^{\ast} = \bigvee_{i = 1}^{n} (ab_{i}a^{\ast})^{\ast}$.
But by axiom (RR5), we have that $(ab_{i}a^{\ast})^{\ast} = (a^{\ast}b_{i}a^{\ast})^{\ast}$.
But this is just $a^{\ast}b_{i}a^{\ast}$ since $b_{i}a^{\ast}$ is a projection.
Thus
$a^{\ast} = \bigvee_{i = 1}^{n} a^{\ast}b_{i}a^{\ast}$.
But $a^{\ast}b_{i}a^{\ast} \leq b_{i}$.
We have therefore proved that $a^{\ast}$ is finite.
We now prove the converse.
Suppose that $a^{\ast}$ is finite.
We prove that $a$ is finite.
Suppose that $a \leq \bigvee_{i \in I} b_{i}$.
From the definition of the natural partial order,
we have that $a = \bigvee_{i \in I} b_{i}a_{i}^{\ast}$.
Thus $b_{i}a^{\ast} \leq a$.
This means that $b_{i}a^{\ast} = ab_{i}^{\ast}a^{\ast}$.
From $a = \bigvee_{i \in I} b_{i}a_{i}^{\ast}$
we deduce that 
$a^{\ast} = \bigvee_{i \in I} b_{i}^{\ast}a_{i}^{\ast}$
by part (2) of Lemma~\ref{lem:posy}.
But $a^{\ast}$ is finite.
Thus, relabelling if necessary, we have that 
$a^{\ast} = \bigvee_{i = 1}^{n} b_{i}^{\ast}a_{i}^{\ast}$.
Multiply on both sides on the left by $a$ to obtain
$a = \bigvee_{i = 1}^{n} ab_{i}^{\ast}a_{i}^{\ast}$.
With what we said above, this proves that $a$ is finite.

(2) Suppose that $a$ and $b$ are both finite and $a \sim_{l} b$.
We prove that $a \vee b$ is finite.
Suppose that $a \vee b \leq \bigvee_{i \in I} a_{i}$.
Since $a \leq a \vee b$ there is a finite subset $I_{1}$ of $I$ such that $a \leq \bigvee_{i \in I_{1}} a_{i}$.
Likewise, there is a finite subset $I_{2}$ of $I$ such that $b \leq \bigvee_{i \in I_{2}} a_{i}$.
It follows that $a \vee b \leq \bigvee_{i \in I_{1} \cup I_{2}} a_{i}$,
and so $a \vee b$ is finite.

(3) This now follows from (1) and (2) above. 
\end{proof}

There is no guarantee that the product of finite elements is finite
but, as we shall see, this will hold in the case of interest to us.

%%%%%%%%%%%%%%%%%%%%%%%%%%%%%%%%%%%%%%%%%%%%%%%%%%%%%%%%%%%%%%%%%%%%%%%%%%%%%%%%%%%%%%%%%%%%%%%%%%%%%%%%%%%%%%%
\section{The structure of \'etale Boolean right restriction monoids}

The goal of this section is to show how to construct \'etale Boolean right restriction monoids from Boolean inverse monoids \cite{Wehrung}.

\begin{remark}{\em Observe that if $S$ and $T$ are isomorphic as inverse semigroups,
then the partially ordered sets $(S,\leq)$ and $(T,\leq)$ are order isomomorphic.
We shall use this observation below in the course of the proof of our first theorem.}
\end{remark}

Our first theorem below shows that the structure of an  \'etale Boolean right restriction monoid is completely
determined by the structure of its Boolean inverse monoid of partial units.
We have rephrased this result slightly at the suggestion of the referee.

\begin{theorem}\label{them:one} Let $S$ and $T$ be \'etale Boolean right restriction monoids.
\begin{enumerate}

\item Every homomorphism of Boolean inverse monoids $\theta \colon \mathsf{Inv}(S) \rightarrow \mathsf{Inv}(T)$ 
induces a homomorphism of Boolean right restriction monoids $\phi \colon S \rightarrow T$ which is the unique extension of $\theta$.

\item If $\mathsf{Inv}(S) \cong \mathsf{Inv}(T)$ then $S \cong T$ as right restriction semigroups.

\end{enumerate}
\end{theorem}
\begin{proof} (1) 
Let $a \in S$. 
Then, under the assumption that $S$ is \'etale, we may write 
$$a = \bigvee_{i=1}^{m} a_{i},$$
where the $a_{i}$ are partial units such that $a_{i} \sim_{l} a_{j}$.
From $a_{i} \sim_{l} a_{j}$, it follows that $\theta (a_{i}) \sim_{l} \theta (a_{j})$
and so $\bigvee_{i=1}^{m} \theta (a_{i})$ is defined in $T$.
This means that we must define 
$$\phi (a) = \bigvee_{i=1}^{m} \theta (a_{i}).$$
Observe that this establishes uniqueness.
Of course, this appears to depend on our choice of partial units $a_{i}$.
We prove that this is not the case.
Accordingly, suppose that $a = \bigvee_{j=1}^{n} b_{j}$,
where the $b_{j}$ are partial units.
Then
$$\bigvee_{i=1}^{m} a_{i} =  \bigvee_{j=1}^{n} b_{j}.$$
Thus $a_{i} \leq  \bigvee_{j=1}^{n} b_{j}$.
It follows by Lemma~\ref{lem:posy}, that
$a_{i} =  \bigvee_{j=1}^{n} a_{i} \wedge b_{j}$.
Now, $a_{i} \wedge b_{j}$ is a partial unit by Lemma~\ref{lem:sleet}.
In addition, $a_{i} \wedge b_{j} \leq a_{i}$.
It follows that these elements are pairwise compatible (not merely left-compatible).
Thus the join is an honest-to-goodness join in $\mathsf{Inv}(S)$.
It follows that $\theta (a_{i}) =  \bigvee_{j=1}^{n} \theta (a_{i} \wedge b_{j})$.
But the elements $a_{i}$ and $b_{j}$ are partial units and we know that they are $\sim_{l}$-related
(because they are all below $a$).
It follows that the meet $a_{i} \wedge b_{j}$ is algebraic by
Lemma~\ref{lem:meets} and the remark which follows.
Thus $\theta (a_{i} \wedge b_{j}) = \theta (a_{i}) \wedge \theta (b_{j})$
and so $\theta (a_{i}) =  \bigvee_{j=1}^{n} \theta (a_{i}) \wedge \theta (b_{j})$.
By Lemma~\ref{lem:posy}, we deduce that $\theta (a_{i}) \leq \bigvee_{j=1}^{n} \theta (b_{j})$.
Thus $\bigvee_{i=1}^{m} \theta (a_{i}) \leq  \bigvee_{j=1}^{n} \theta (b_{j})$.
By symmetry we get equality.
This proves that $\phi$ is a well-defined function extending $\theta$.
The fact that $\phi$ is a semigroup homomorphism follows from the observation that if
$a = \bigvee_{i=1}^{m} a_{i}$ and $b = \bigvee_{j=1}^{n}b_{j}$,
where the $a_{i}$ and $b_{j}$ are partial units,
then $ab = \bigvee_{i,j} a_{i}b_{j}$,
where each product $a_{i}b_{j}$ is a partial unit by Lemma~\ref{lem:sleet}.
Suppose that $a = \bigvee_{i=1}^{m}a_{i}$, where the $a_{i}$ are partial units.
Then we claim that $a^{\ast} =  \bigvee_{i=1}^{m}a_{i}^{-1}a_{i}$.
From 
$a = \bigvee_{i=1}^{m}a_{i}$
we get that
$a^{\ast} = \bigvee_{i=1}^{m}a_{i}^{\ast}$
by part (2) of Lemma~\ref{lem:posy}. 
But
$a_{i}^{\ast} = a_{i}^{-1}a_{i}$
by Lemma~\ref{lem:spilled-tea} and the sentences which follow.
We have proved that $\phi$ is a homorphism of right restriction monoids.

(2) Suppose, now, that $\theta$ is a bijection.
We prove that $\phi$ is a bijection.
Let $a,b \in S$ such that $\phi (a) = \phi(b)$.
We can write
$$a = \bigvee_{i=1}^{m} a_{i} \text{ and } b = \bigvee_{j=1}^{n} b_{j},$$
where the $a_{i}$ and $b_{j}$ are partial units.
By assumption
$$\bigvee_{i=1}^{m} \theta (a_{i}) = \bigvee_{j=1}^{n} \theta (b_{j}).$$
Thus $\theta (a_{i}) \leq  \bigvee_{j=1}^{n} \theta (b_{j})$.
It follows by Lemma~\ref{lem:posy},
that 
$\theta (a_{i}) =  \bigvee_{j=1}^{n} \theta (a_{i}) \wedge \theta (b_{j})$.
The element $\theta (a_{i}) \wedge \theta (b_{j})$ is a partial unit by Lemma~\ref{lem:holdo}.
Since $\theta$ is an isomorphism (and so an order isomorphism)
it follows that $a_{i} \wedge b_{j}$ exists and
$\theta (a_{i}) \wedge \theta (b_{j}) = \theta (a_{i} \wedge b_{j})$.
Again, since $\theta$ is an isomorphism
$\bigvee_{j=1}^{n} \theta (a_{i} \wedge b_{j}) = \theta \left(\bigvee_{j=1}^{n} a_{i} \wedge b_{j} \right)$.
It follows that $a_{i} =  \bigvee_{j=1}^{n} a_{i} \wedge b_{j}$.
Whence $a_{i} \leq b$.
This means that $a \leq b$.
By symmetry we have that $a = b$ and we have proved that $\phi$ is injective.
We now prove that $\phi$ is surjective.
Let $a' \in T$.
Then $a' = \bigvee_{i=1}^{m} a_{i}'$
where the $a_{i}'$ are partial units which are pairwise left-compatible.
Let $a_{i}$ be the partial unit of $S$ such that $\theta (a_{i}) = a_{i}'$.
The elements $a_{i}$ are also left-compatible.
Put $a = \bigvee_{i=1}^{m} a_{i}$.
Then, by construction, $\phi (a) = a'$. 
We have therefore proved that $\phi$ is an isomorphism of right restriction monoids.
\end{proof}

%%%%%%%%%%%%%%%%%%%%%%%%%%%%%%%%%%%%%%%%%%%%%%%%%%%%%%%%%%%%%%%%%%%%%%%%%%%%%%%%%%%%%%%%%%%%%%%%%%%%%%%%
Theorem~\ref{them:one} is purely theoretical in that it shows that Boolean inverse monoids 
determine the structure of \'etale Boolean right restriction monoids.
We now show how to actually construct all \'etale Boolean right restriction monoids
directly from Boolean inverse monoids. 
We shall apply the theory developed in Section~4. 

%%%%%%%%%%%%%%%%%%%%%%%%%%%%%%%%%%%%%%%%%%%%%%%%%%%%%%%%%%%%%%%%%%%%%%%%%%%%%%%%%%%%%%%%%%%%%%%%%%%%%%%%%%%%%%%%%%%%%
We begin with some motivation.
Let $S$ be a Boolean inverse monoid.
For each $a \in S$ define $[a] = a^{\downarrow} \cap \mathsf{Inv}(S)$.

\begin{lemma}\label{lem:regis} Let $S$ be an \'etale Boolean right restriction monoid.
\begin{enumerate}
\item The set $[a]$ is acceptable.

\item The finite compatible join of partial units in $[a]$ is still in $[a]$.

\item There is a finite set of left-compatible partial units $\{b_{1},\ldots ,b_{n}\}$
in $[a]$ such that if $c \in [a]$ 
then $c$ is a compatible join of partial units in  $\{b_{1},\ldots ,b_{n}\}^{\downarrow}$.

\item $[a] = [b]$ implies that $a = b$.

\end{enumerate}
\end{lemma}
\begin{proof} (1) $a^{\downarrow}$ is an order ideal by construction,
and $\mathsf{Inv}(S)$ is an order ideal by part (2) of Lemma~\ref{lem:sleet}.
Thus  $[a]$ is an order ideal.
The fact that this set is left-compatible follows by Lemma~\ref{lem:spiked}.

(2) Suppose that $\{a_{1},\ldots, a_{n}\}$ is a compatible set of partial units each less
than or equal to $a$.
Then their join is less than or equal to $a$.
This join is a partial unit by Lemma~\ref{lem:sleet}.

(3) By the definition of \'etale, we may write $a = \bigvee_{i=1}^{n} b_{i}$,
where the $b_{i}$ are a finite left-compatible set of partial units.
Let $c$ be any partial unit $c \leq a$.
Then $c = \bigvee_{i=1}^{n} c \wedge b_{i}$ by Lemma~\ref{lem:posy}.
By part (2) of Lemma~\ref{lem:sleet}, each  $c \wedge b_{i}$ is a partial unit
and $c \wedge b_{i} \leq b_{i}$.

(4) We prove first that $a \leq b$;
the result then follows by symmetry.
By defintion, $a = \bigvee_{i=1}^{m} a_{i}$ where the $a_{i} \in a^{\downarrow} \cap \mathsf{Inv}(S)$
and so are elements of $b^{\downarrow} \cap \mathsf{Inv}(S)$.
It follows that $a_{i} \leq b$ from which it follows that $a \leq b$.
\end{proof}

We have therefore proved that the sets $[a]$ are special elements of $\mathsf{R}(S)$.
We shall prove that they are precisely the finite elements which are closed with respect to a suitable nucleus.

%%%%%%%%%%%%%%%%%%%%%%%%%%%%%%%%%%%%%%%%%%%%%%%%%%%%%%%%%%%%%%%%%%%%%%%%%%%%%%%%%%%%%%%%%%%%%%%%%%%%%%%%%%%%%%%%%%%%%%%%%%%%%%%%%%%
\begin{quote}
{\em Our goal is to do the following: given a Boolean inverse monoid $S$,
we shall construct an \'etale Boolean right restriction monoid $\mathsf{Etale}(S)$,
called the {\em (right restriction) companion} of $S$, such that $S \cong \mathsf{Inv}(\mathsf{Etale}(S))$.}
\end{quote}

Let $X \subseteq S$. 
Define $X^{\vee}$ to be the set of all {\em compatible}
finite joins of elements of $X$.
We say that $X$ is {\em closed} if $X = X^{\vee}$.

\begin{lemma}\label{lem:coffee} 
Let $S$ be a Boolean inverse monoid.
\begin{enumerate}

\item  If $A$ is acceptable in $S$ then $A^{\vee}$ is acceptable in $S$ and $A \subseteq A^{\vee}$.

\item The function $A \mapsto A^{\vee}$ is a nucleus on  $\mathsf{R}(S)$.

\end{enumerate}
\end{lemma}
\begin{proof} (1) We prove that $A^{\vee}$ is acceptable.
We prove first that  $A^{\vee}$ is an order ideal.
A typical element of  $A^{\vee}$
has the form $\bigvee_{i = 1}^{m} a_{i}$ 
where $\{a_{1}, \ldots, a_{m}\}$ is a compatible subset of $A$.
Suppose that $x \leq \bigvee_{i = 1}^{m} a_{i}$.
So, by Lemma~\ref{lem:posy}, we have that  
$x = \bigvee_{i = 1}^{m} a_{i} \wedge x$.
Because $A$ is an order ideal, we have that $a_{i} \wedge x \in A$.
Thus $x \in A^{\vee}$.

We prove next that  $A^{\vee}$ is left-compatible.
Suppose that $x,y \in A^{\vee}$.
Then $x = \bigvee_{i=1}^{m}x_{i}$ and $y = \bigvee_{j=1}^{n} y_{j}$ 
where $x_{i}, y_{j} \in A$.
We are working inside a Boolean inverse monoid and so
$y^{-1} = \bigvee_{j=1}^{n} y_{j}^{-1}$. 
By assumption, $x_{i} \sim_{l} y_{j}$ for all $i$ and $j$.
Thus by Lemma~\ref{lem:yakult},
we have that $x_{i}y_{j}^{-1}$ is an idempotent.
But $xy^{-1} = \bigvee_{i,j}x_{i}y_{j}^{-1}$.
In a Boolean inverse monoid the finite join of idempotents is an idempotent and so $xy^{-1}$ is an idempotent.
By Lemma~\ref{lem:yakult},
we have proved that $x$ and $y$ are left-compatible.
Thus $A^{\vee}$ is acceptable.
It is immediate that $A \subseteq A^{\vee}$.

(2) The function is well-defined by part (1) above.
The proof that the properties (N1), (N2) and (N3) of a nucleus hold are immediate.
It is easy to verify that (N4) and (N5) hold.
We prove that (N6) holds.
Only one direction needs proving.
Let $x \in ((A^{\vee})^{\ast})^{\vee}$.
Then $x = \bigvee_{i=1}^{m} x_{i}$
where $x_{i} \in (A^{\vee})^{\ast}$.
Now, each $x_{i} = (\bigvee_{j_{i}} y_{j_{i}})^{\ast}$ where $y_{j_{i}} \in A$.
By part (2) of Lemma~\ref{lem:posy},
we have that 
$x_{i} = \bigvee_{j_{i}} y_{j_{i}}^{\ast}$.
Thus $x = \bigvee_{i}\bigvee_{j_{i}} y_{j_{i}}^{\ast}$ where $y_{j_{i}} \in A$.
Whence $x \in (A^{\ast})^{\vee}$.
\end{proof}

Now define $\nu (A) = A^{\vee}$.
Then we have the following by Proposition~\ref{prop:as} and Lemma~\ref{lem:coffee}.

\begin{lemma}\label{lem:rubin} Let $S$ be a Boolean inverse monoid.
Then $\mathsf{R}(S)_{\nu}$ is a complete right restriction monoid.
\end{lemma}

If we go back to Lemma~\ref{lem:regis},
the sets $[a]$ are closed with respect to the nucleus we have defined above. 

%%%%%%%%%%%%%%%%%%%%%%%%%%%%%%%%%%%%%%%%%%%%%%%%%%%%%%%%%%%%%%%%%%%%%%%%%%%%%%%%%%%%%%%%%%%%%%%%%%%%%
The monoid  $\mathsf{R}(S)_{\nu}$, constructed in Lemma~\ref{lem:rubin}, is too big for our purposes
and so we shall cut it down.
Define 
$$\mathsf{Etale}(S) = \mathsf{fin}(\mathsf{R}(S)_{\nu}),$$
the finite elements of $\mathsf{R}(S)_{\nu}$.
{\em Recall that we are using the word `finite' in the technical sense as defined at the conclusion of Section~4.
The reader might wish to substitute the word `compact' for `finite' in what follows.}
Our first job, therefore, is to describe  the finite elements in $\mathsf{R}(S)_{\nu}$.

\begin{lemma}\label{lem:STJ} Let $S$ be a Boolean inverse monoid.
The finite elements of $\mathsf{R}(S)_{\nu}$  are the subsets of the form 
$(\{a_{1}, \ldots, a_{m}\}^{\downarrow})^{\vee}$,
where $\{a_{1}, \ldots, a_{m}\}$ is a left-compatible subset of $S$.
\end{lemma} 
\begin{proof} The proof of Proposition~\ref{prop:as} will be used below:
particularly the form taken by the natural partial order and the join.

We work in the complete right restriction monoid $\mathsf{R}(S)_{\nu}$ defined above.
Observe that for each $a \in S$, the set $a^{\downarrow}$ belongs to $\mathsf{R}(S)_{\nu}$.
We show that, in fact, each element of the form $a^{\downarrow}$, where $a \in S$, is finite (in the technical sense).
Suppose that $a^{\downarrow} \leq \bigsqcup_{i \in I} B_{i}$  
where $B_{i} \in  \mathsf{R}(S)_{\nu}$.
Then $a^{\downarrow} \subseteq \nu ( \bigvee_{i \in I} B_{i})$  
where $B_{i}$ are $\nu$-closed elements of $\mathsf{R}(S)$.
Thus $a = \bigvee_{j=1}^{n} b_{j}$ where $b_{j} \in B_{j}$,
relabelling if necessary.
This implies that $a^{\downarrow} \leq \bigsqcup_{j = 1}^{n} B_{i}$.
This proves that $a^{\downarrow}$ is finite.

Put $A = \{a_{1}, \ldots, a_{m}\}$. 
Then $A^{\downarrow}$ is an acceptable set, using Lemma~\ref{lem:similarity}, and so a well-defined element of  $\mathsf{R}(S)$.
Observe that if $X = (\{a_{1}, \ldots, a_{m}\}^{\downarrow})^{\vee}$,
then $X = a_{1}^{\downarrow} \sqcup \ldots \sqcup a_{1}^{\downarrow}$ 
and so is finite, since any finite join of such elements is again finite
by part (2) of Lemma~\ref{lem:klavin}. 

We now show that all finite elements have this form.
Suppose that $X$ is a finite element of  $\mathsf{R}(S)_{\nu}$.
Thus $X$ is a closed acceptable subset.
We may write $X = \bigsqcup_{a \in A} a^{\downarrow}$.
But we have assumed that $X$ is finite.
Thus $X = \bigsqcup_{i=1}^{m} a_{i}^{\downarrow}$ for some $m$.
\end{proof}

If we go back to Lemma~\ref{lem:regis}, then we see that the elements $[a]$
are indeed finite in our technical sense.

The proof of the following is now easy by Lemma~\ref{lem:STJ}
since we can describe the finite elements of $\mathsf{R}(S)_{\nu}$
and apply Lemma~\ref{lem:similarity}. 

\begin{lemma}\label{lem:tucker} 
In  $\mathsf{R}(S)_{\nu}$, the product of finite elements is a finite element.
\end{lemma}
\begin{proof} According to Lemma~\ref{lem:STJ},
we have a description of the finite elements.
Let $A = (\{a_{1}, \ldots, a_{m}\}^{\downarrow})^{\vee}$,
where $\{a_{1}, \ldots, a_{m}\}$ is a left-compatible subset,
and let 
$B = (\{b_{1}, \ldots, b_{m}\}^{\downarrow})^{\vee}$,
where $\{b_{1}, \ldots, b_{m}\}$ is a left-compatible subset,
be finite elements of $\mathsf{R}(S)_{\nu}$.
Their product, $A \cdot B$ is just
$(\{a_{1}, \ldots, a_{m}\}^{\downarrow}\{b_{1}, \ldots, b_{m}\}^{\downarrow})^{\vee}$,
where we have used the definition of the product in $\mathsf{R}(S)_{\nu}$
and Lemma~\ref{lem:nucleus-properties}.
Now we apply Lemma~\ref{lem:similarity}
to deduce that $a_{i}b_{j} \sim_{l} a_{s}b_{t}$.
It follows that $A \cdot B$ is also finite (in our technical sense). 
\end{proof}

%%%%%%%%%%%%%%%%%%%%%%%%%%%%%%%%%%%%%%%%%%%%%%%%%%%%%%%%%%%%%%%%%%%%%%%%%%%%%%%%%%%%%%%%%%%%%%%%%%%%%%%%%%%%
We can now state and prove the main theorem of this paper.
This shows us how to construct all \'etale Boolean right restriction monoids from Boolean inverse monoids.

\begin{theorem}\label{them:etale-main}  Let $S$ be a Boolean inverse monoid.
Then $\mathsf{Etale}(S)$ is an \'etale Boolean right restriction monoid
whose semigroup of partial units is isomorphic to $S$.
\end{theorem}
\begin{proof} The product of two finite elements is finite by Lemma~\ref{lem:tucker}.
It follows by part (3) of Lemma~\ref{lem:klavin}, that $\mathsf{Etale}(S)$
is a distributive right restriction monoid. 
If $A$ is finite (in our technical sense) then $A^{\circ}$ is finite by Lemma~\ref{lem:klavin}.
However, $A^{\circ}$ is of the form $(\{a_{1}^{\ast},\ldots,  a_{m}^{\ast}\}^{\downarrow})^{\vee}$
which is equal to $(a_{1}^{\ast} \vee \ldots  \vee a_{m}^{\ast})^{\downarrow}$.
We deduce that the projections of $\mathsf{Etale}(S)$ are the principal order ideals of the idempotents of $S$
and so form a meet-semilattice isomorphic with the set of idempotents of $S$, which is a Boolean algebra.
It follows that $\mathsf{Etale}(S)$ is, in fact, a Boolean right restriction monoid.

Elements of the form $a^{\downarrow}$, for some $a \in S$, are partial units
because 
$a^{\downarrow}(a^{-1})^{\downarrow} = (aa^{-1})^{\downarrow}$ 
and
$(a^{-1})^{\downarrow}a^{\downarrow} = (a^{-1}a)^{\downarrow}$.
However,  every element of $\mathsf{Etale}(S)$ can be written $\bigsqcup_{i=1}^{m} a_{i}^{\downarrow}$,
for some left-compatible subset $\{a_{1},\ldots, a_{m}\}$.
This immediately implies that $\mathsf{Etale}(S)$ is \'etale.
We shall now prove that every partial unit  $\mathsf{Etale}(S)$
is of the form  $a^{\downarrow}$ from which it follows that the partial units of 
$\mathsf{Etale}(S)$ form a semigroup isomorphic to $S$.
Let $A = (\{a_{1},\ldots, a_{m}\}^{\downarrow})^{\vee}$ be a partial unit of $\mathsf{Etale}(S)$.
Then there is an element  
$X = (\{x_{1},\ldots, x_{n}\}^{\downarrow})^{\vee}$ of $\mathsf{Etale}(S)$
such that $A \cdot X = (x_{1}^{\ast} \vee \ldots \vee x_{n}^{\ast})^{\downarrow}$ 
and
$X \cdot A = (a_{1}^{\ast} \vee \ldots \vee a_{m}^{\ast})^{\downarrow}$.
Choose any $a_{i}$.
Then $a_{i}^{\ast} \in X \cdot A$.
It follows that $a_{i}^{\ast} = \bigvee_{j=1}^{s}y_{j}b_{j}$
where $y_{j} \in \{x_{1},\ldots, x_{n}\}^{\downarrow}$
and $b_{j} \in \{a_{1},\ldots, a_{m}\}^{\downarrow}$.
Since each  $b_{j} \in \{a_{1},\ldots, a_{m}\}^{\downarrow}$, we may write $b_{j} = b_{j}b_{j}^{-1}a_{j}'$
where $a_{j}'$ is one of the elements $a_{1},\ldots, a_{m}$.
It is therefore immediate that $y_{j}b_{j} = (y_{j}b_{j}b_{j}^{-1})a_{j}'$.
But $y_{j}b_{j}b_{j}^{-1} \in \{x_{1},\ldots, x_{n}\}^{\downarrow}$ since this is an order ideal.
Thus we can write (relabelling if necessary)
$a_{i}^{-1}a_{i} = \bigvee_{j=1}^{s} y_{j}a_{j}'$
where $y_{j} \in  \{x_{1},\ldots, x_{n}\}^{\downarrow}$ 
and the $a_{j}'$ is one of the elements $a_{1},\ldots, a_{m}$.
It follows that
$a_{i}^{-1} =  \bigvee_{j=1}^{s} y_{j}(a_{j}'a_{i}^{-1})$
since $a_{i}^{-1} = a_{i}^{-1}a_{i}a_{i}^{-1}$.
Because $\{a_{1},\ldots, a_{m}\}$ is a left-compatible set
the element
$a_{j}'a_{i}^{-1}$ is always a projection/idempotent by Lemma~\ref{lem:yakult}.
Thus  $y_{j}(a_{j}'a_{i}^{-1}) \in \{x_{1},\ldots, x_{n}\}^{\downarrow}$.
We have therefore shown that $a_{i}^{-1} \in X$.
The elements of $X$ are left-compatible.
It follows that $a_{i}^{-1}$ and $a_{j}^{-1}$ are left-compatible,
and so, $a_{i}$ and $a_{j}$ are right-compatible by Lemma~\ref{lem:yoghurt}.
Thus $a_{i}$ and $a_{j}$ are compatible.
It follows that $A = (\{a_{1},\ldots, a_{m}\}^{\downarrow})^{\vee} = (a_{1} \vee \ldots \vee a_{m})^{\downarrow}$.
\end{proof}

The following is included at the suggestion of the referee and combines
Theorem~\ref{them:one} and Theorem~\ref{them:etale-main}.

\begin{theorem}\label{thm:referee} Every \'etale Boolean right restriction monoid $S$
is the companion of a unique (upto isomorphism) Boolean inverse monoid:
namely, $\mathsf{Inv}(S)$.
\end{theorem}

\begin{remark}\label{rem:functor}
{\em This remark is included at the suggestion of the referee.
Denote by $\mathscr{B}$ the category of Boolean inverse monoids and their homomorphisms.
Denote by $\mathscr{E}$ the category of \'etale Boolean right restriction monoids and their homomorphisms.
By Lemma~\ref{lem:functor}, define a functor $F \colon \mathscr{E} \rightarrow \mathscr{B}$,
which associates with every \'etale Boolean right restriction monoid $S$
its Boolean inverse monoid of partial units $\mathsf{Inv}(S)$,
and with every $\theta \colon S \rightarrow T$, a homomorphism of Boolean 
right restriction monoids, its restriction to the partial units of $S$.
Define a functor $G \colon \mathscr{B} \rightarrow \mathscr{E}$, 
which associates with each Boolean inverse monoid $S$ the
\'etale Boolean right restriction monoid $\mathsf{Etale}(S)$,
and with every homomorphism from $S$ to $T$
the homomorphism from $\mathsf{Etale}(S)$ to $\mathsf{Etale}(T)$
guaranteed by part (1) of Theorem~\ref{them:one}.
In fact, these functors induce an equivalence of categories.
}
\end{remark}

\begin{remark}{\em There is a simpler way of constructing $\mathsf{Etale}(S)$ but at the expense
of using non-commutative Stone duality \cite{Lawson2024}.
This construction does not use anything we have hitherto introduced in this paper,
and will not be used below.
It is included only for the sake of completeness.
If $G$ is a groupoid then $G_{o}$ denotes its space of identities.
If $g \in G$ then we write $\mathbf{d}(g) = g^{-1}g$ and $\mathbf{r}(g) = gg^{-1}$.
Our convention is that $gh$ is defined in a groupoid precisely when $\mathbf{d}(g) = \mathbf{r}(h)$.
Classical Stone duality tells us that every Boolean algebra gives rise to a Boolean space
and the clopen subsets of a Boolean space form a Boolean algebra.
A {\em Boolean groupoid} is an \'etale topological groupoid whose space of identities is a Boolean space.
A subset $A \subseteq G$ of a groupoid is said to be a {\em local bisection} if $A^{-1}A,AA^{-1} \subseteq G_{o}$.
The set of compact-open local bisections, $\mathsf{KB}(G)$, of a Boolean groupoid  $G$ forms a Boolean inverse monoid
under subset multiplication,
and it is the substance of non-commutative Stone duality that every Boolean inverse monoid arises in this way.

We shall now generalize this idea.
Let $G$ be a groupoid.
We say that $A \subseteq G$ is a {\em local right bisection}
if $a,b \in A$ and $\mathbf{d}(a)  = \mathbf{d}(b)$ implies that $a = b$.
It is easy to check that the product of local right bisections is again a local right bisection.
Denote by $\mathsf{R}(G)$ the set of all compact-open local right bisections of the Boolean groupoid $G$.
Let $A$ and $B$ be compact-open local right bisections.
Then $AB$ is open, because $G$ is an \'etale groupoid, and we have already noted that it is a local right bisection.
It is compact because the multiplication map in a topological groupoid is continuous.
It is thus easy to see that $\mathsf{R}(G)$ is a semigroup; in fact, it is a monoid.
If $A \in \mathsf{R}(G)$, define $A^{\ast} = \{\mathbf{d}(a) \colon a \in A \}$.
This is a clopen subset of the space $G_{o}$ because $G$ is \'etale. 
Put $\mathsf{Proj} (\mathsf{R}(G))$ equal to the set of all clopen subsets of the space $G_{o}$.
This is a Boolean algebra.
Using, in particular,  the properties of $\mathbf{d}$, it is now routine to check that $\mathsf{R}(G)$ is a Boolean right restriction monoid.
The partial units of this monoid are precisely the compact-open local bisections.
The monoid $\mathsf{R}(G)$ is \'etale because $G$ has a basis consisting of compact-open local bisections.
This proves that $\mathsf{R}(G)$ is the companion of $\mathsf{KB}(G)$.

We now describe an example of the above procedure.
Let $X$ be a finite non-empty set equipped with the discrete topology.
We return to Example~\ref{ex:first}.
The monoid
$\mathsf{PT}(X)$ is a Boolean right restriction monoid.
We locate the partial units.
We claim that these are precisely the elements $\mathsf{I}(X)$,
the set of partial bijections on $X$.
It is clear that $\mathsf{I}(X) \subseteq \mathsf{Inv}(X)$.
Let $f \in \mathsf{PT}(X)$ be a partial unit.
Then there is an element $g \in \mathsf{PT}(X)$
such that $fg = g^{\ast}$ and $gf = f^{\ast}$.
Suppose that $f(x) = f(y)$ where $x,y \in \mbox{dom}(f)$.
Then $g(f(x)) = g(f(y))$.
But $gf$ is the identity function on $\mbox{dom}(f)$.
Thus $x = y$.
We have proved that $f$ is injective.
Put $\mbox{im}(f) = \mbox{dom}(g)$.
Let $y \in \mbox{dom}(g)$.
Then $(fg)(y) = y$.
Put $x = g(y)$.
Then $f(x) = (fg)(y) = y$.
We therefore have a bijection from $\mbox{dom}(f)$ onto $\mbox{dom}(g)$.
It follows that $f \in \mathsf{I}(X)$.
Define the partial function $g^{x}_{y}$ which has domain of definition $\{x\}$ and maps $x$ to $y$.
Clearly,  $g^{x}_{y}$ is a partial unit
and each element of $\mathsf{PT}(X)$ is a finite join of left-compatible elements
of the form $g^{x}_{y}$.
Thus $\mathsf{PT}(X)$ is \'etale.
It follows that in the case where $X$ is finite,
$\mathsf{PT}(X)$ arises from $\mathsf{I}(X)$ via the construction of this paper.
This fits into the procedure described above.
The Boolean groupoid associated with $\mathsf{I}(X)$ is $X \times X$,
where $(x,y)^{-1} = (y,x)$ and $(x,y)(y,z) = (x,z)$.
The compact-open local right bisections of $X \times X$ correspond exactly  to the partial functions of $X$.}
\end{remark}

We shall now describe a special case of our categorical equivalence presented in Remark~\ref{rem:functor}.
An {\em additive ideal} in a Boolean inverse semigroup is an ideal which is closed under finite compatible joins.
We say that a Boolean inverse monoid $S$ is {\em $0$-simplifying}
if the only additive ideals are $\{0\}$ and $S$.
The following is a version of \cite[Theorem 4.7]{G2}.

\begin{lemma}\label{lem:ideals} Let $S$ be an \'etale Boolean right restriction monoid.
Then $\mathsf{Inv}(S)$ is $0$-simplifying if and only if
for each non-zero projection $e$ there exists a total  element $a \in S$ such that $(ea)^{\ast} = 1$.
\end{lemma}
\begin{proof} Suppose first that $I = \mathsf{Inv}(S)$ is $0$-simplifying.
Let $e$ be any non-zero projection.
Then $IeI$ is an ideal of $I$ and $(IeI)^{\vee}$ is an additive ideal of $I$.
By assumption, $(IeI)^{\vee} = I$.
Since $1 \in I$, it follows that $1 = \bigvee_{i=1}^{n} e_{i}$
for some idempotents/projections $e_{i} \in IeI$. 
Observe that the join $e_{1} \vee e_{2} \overline{e_{1}} \vee \ldots \vee e_{n} \overline{e_{1}} \, \overline{e_{2}} \ldots \, \overline{e_{n-1}}$
is also $1$ and that each term belongs to $IeI$ since this is an order ideal.
It follows that we lose no generality in assuming that $1 = \bigvee_{i=1}^{n} e_{i}$ is an orthogonal join.
By assumption, $e_{i} = a_{i}eb_{i}$ where $a_{i}, b_{i} \in \mathsf{Inv}(S)$.
Define $x_{i} = e\mathbf{d}(a_{i})b_{i}$.
Observe that $x_{i} = ex_{i}$.
A simple calculation show that $x_{i}^{-1}x_{i} = e_{i}$.
Thus $1 = \bigvee_{i=1}^{n} x_{i}^{-1}ex_{i}$
Observe that the elements $x_{i}$ are pairwise left-orthogonal.
Put $a = \bigvee_{i=1}^{n} x_{i} \in S$.
But $a^{\ast} = 1$.
We now calculate $(ea)^{\ast}$.
But this is just $1$.
We now prove the converse.
Let $J$ be any non-trivial additive ideal of $\mathsf{Inv}(S)$
and let $e$ be a non-zero projection in $J$.
Them, by assumption, there is an element $a \in S$ such that $(ea)^{\ast} = 1$.
By assumption, $S$ is \'etale.
Thus there are partial units $a_{i}$ such that $a = \bigvee_{i=1}^{n} a_{i}$.
Thus $1 = (ea)^{\ast} = \bigvee_{i=1}^{n} (ea_{i})^{\ast}$.
It follows that $1 \in J$ and so $J = \mathsf{Inv}(S)$.
We have proved that $\mathsf{Inv}(S)$ is $0$-simplifying.
\end{proof}

\begin{remark}
{\em Let $S$ be a Boolean inverse monoid.
There is a {\em natural (right) action} of $\mathsf{U}(S)$, the group of units of $S$, on $\mathsf{E}(S)$, the Boolean algebra of idempotents of $S$, given by
$e \mapsto g^{-1}eg = (eg)^{\ast}$.
An inverse semigroup is said to be {\em fundamental}  
if the only element that commutes with all idempotents is itself an idempotent.
By \cite[Lemma 5.5]{Lawson2016},  a Boolean inverse monoid is fundamental iff the natural action is faithful.
A Boolean inverse monoid is said to be {\em additively simple}
if it is both $0$-simplifying and fundamental.
Those \'etale Boolean right restriction monoids $S$ are most interesting
where $\mathsf{Inv}(S)$ is additively simple.
}
\end{remark}

%%%%%%%%%%%%%%%%%%%%%%%%%%%%%%%%%%%%%%%%%%%%%%%%%%%%%%%%%%%%%%%%%%%%%%%%%%%%%%%%%%%%%%%%%%%%%%%%%%%%%%%%%%%%%%%%%%%%%%%%%%%%%%%%
We can now determine for which class of Boolean inverse monoids $S$ the companion is actually isomorphic to $S$.

\begin{proposition} Let $S$ be a Boolean inverse monoid.
Then $\mathsf{Etale}(S) \cong S$ if and only if in $S$ left-compatible elements are compatible.
\end{proposition}
\begin{proof} Suppose first that in $S$ left-compatible elements are compatible.
Then it is immediate from the way that $\mathsf{Etale}(S)$ is constructed that
$\mathsf{Etale}(S) \cong S$. 
We now prove the converse.
Suppose that $\mathsf{Etale}(T) \cong S$.
Thus, we are given an \'etale Boolean right restriction monoid $T$ such that
$T \cong \mathsf{Inv}(S)$.
Thus $T$ is an  \'etale Boolean right restriction monoid which is also inverse.
We prove first that $T = \mathsf{Inv}(T)$.
We therefore need to prove that every element of $T$ is a partial unit.
Let $a \in T$.
Then there is a unique element $b \in T$ such that
$a = aba$ and $b = bab$.
By assumption
$$a = \bigvee_{i=1}^{m} a_{i} \text{ and } b = \bigvee_{j=1}^{m} b_{j},$$
where the $a_{i}$ and $b_{j}$ are partial units.
It follows that
$$ab =  \bigvee_{1 \leq i \leq m, 1 \leq j \leq n} a_{i}b_{j},$$
where $a_{i}b_{j}$ is a partial unit by part (1) of Lemma~\ref{lem:holdo}.
But $ab$ is an idempotent and so $a_{i}b_{j}$ is an idempotent.
By part (3) of  Lemma~\ref{lem:holdo}, it follows that  $a_{i}b_{j}$ 
is a projection. 
By part (2) of Lemma~\ref{lem:posy}, it follows that $ab$ is a projection.
By symmetry, $ba$ is a projection.
Thus $a$ is a partial unit.
We have therefore proved that $T = \mathsf{Inv}(T)$.
We can now finish off the proof.
Suppose that $a,b \in T$ are such that $a \sim_{l} b$.
Then $a \vee b$ exists and, by assumption, is a partial unit.
We have that $a,b \leq a \vee b$.
It follows by part (2) of Lemma~\ref{lem:holdo}, that
$a = aa^{-1}(a \vee b)$.
By Lemma~\ref{lem:posy}, we have that
$a = a \vee aa^{-1}b$.
Thus, in particular, $aa^{-1}b \leq a$.
But $bb^{-1}a \leq a$.
Now check that $(aa^{-1}b)^{\ast} = (bb^{-1}a)^{\ast}$,
using the fact that projections commute.
By part (1) of Lemma~\ref{lem:spiked}, we have that $aa^{-1}b = bb^{-1}a$.
Thus $a \sim_{r} b$.
We have therefore proved that $a \sim b$.
\end{proof}

The above result therefore highlights the class of inverse semigroups in which $\sim_{l} \, \subseteq \, \sim_{r}$.
The $E$-reflexive inverse semigroups, discussed in \cite[page 86]{Lawson1998}, are examples.

%%%%%%%%%%%%%%%%%%%%%%%%%%%%%%%%%%%%%%%%%%%%%%%%%%%%%%%%%%%%%%%%%%%%%%%%%%%%%%%%%%%%%%%%%%%%%%%%%%%%%%%%%%%%%%%%%%%%%%
\section{Examples}

In this section, we shall illustrate the theory developed in this paper by describing two examples.

%%%%%%%%%%%%%%%%%%%%%%%%%%%%%%%%%%%%%%%%%%%%%%%%%%%%%%%%%%%%%%%%%%%%%%%%%%%%%%%%%%%%%%%%%%%%%%%%%%%%%%%%%%%%%%%%%
\subsection{Partial homeomorphisms of a Boolean space}

Let $X$ be a Boolean space.
Denote by $\mathsf{I}^{\tiny cl}(X)$ the set of all partial homeomorphisms between the clopen sets of $X$.
This is a Boolean inverse monoid by \cite[Proposition 2.16, Proposition 5.2]{Lawson2016}.
Let  $\mathsf{S}(X)$ be the set of all local homeomorphisms $\theta \colon U \rightarrow X$
where $U$ is clopen.
This is a Boolean right restriction monoid \cite[Example 5.6]{Lawson2021b}.
Observe that $U$ is compact and so $\theta (U)$ is compact.
But $X$ is Hausdorff.
It follows that $\theta (U)$ is also clopen. 
Thus  $\mathsf{S}(X)$ is the set of all surjective local homeomorphisms
between the clopen subsets of $X$.
Clearly, $\mathsf{I}^{\tiny cl}(X) \subseteq \mathsf{S}(X)$ and
the set of partial units of $\mathsf{S}(X)$ is $\mathsf{I}^{\tiny cl}(X)$.
Let $\theta \colon U \rightarrow V$ be a surjective local homeomorphism
between two clopen sets.
For each $x \in U$, there is an open set $U_{x} \subseteq U$ such that
$\theta$ restricted to this set is a homeomorphism.
But the clopens form a basis for the topology on $X$.
So, without loss of generality, we may assume that $U_{x}$ is clopen.
Thus, the $U_{x}$ cover $U$.
But $U$ is compact.
It follows that we can write $\theta$ as a finite union of partial homeomorphisms.
We have therefore shown that $\mathsf{S}(X)$ is \'etale.
Since the set of partial units of $\mathsf{S}(X)$ is  $\mathsf{I^{\tiny cl}}(X)$,
it follows by Theorem~\ref{them:one} that  $\mathsf{S}(X)$ is the companion of  $\mathsf{I}^{\tiny cl}(X)$;
that is, $\mathsf{S}(X) = \mathsf{Etale}(\mathsf{I}^{\tiny cl}(X))$.

Although not needed, we shall show how to abstractly construct  $\mathsf{S}(X)$ from $\mathsf{I}^{\tiny cl}(X)$.
Let $A = (\{f_{1},\ldots, f_{m}\}^{\downarrow})^{\vee}$ be a finite closed acceptable set in $\mathsf{I}^{\tiny cl}(X)$.
We have that $f_{1},\ldots, f_{m} \in \mathsf{I}(X)$
and to say that they are left-compatible simply means that $\theta = f_{1} \cup \ldots \cup f_{m}$ is a well-defined
partial function of $X$.
Put $U = U_{1} \cup \ldots \cup U_{n}$ where $U_{i}$ is the domain of definition of $f_{i}$.
Put $V = V_{1} \cup \ldots \cup V_{n}$ where $V_{i}$ is the range of $f_{i}$.
The sets $U$ and $V$ are both clopen and $\theta \colon U \rightarrow V$
is a surjective local homeomorphism.
We have therefore defined an element of $\mathsf{S}(X)$.
We show that $A = \theta^{\downarrow} \cap  \mathsf{I}(X)$.
Let $f \in A$.
Then $f$ is a compatible join of elements of 
$\{f_{1},\ldots, f_{m}\}^{\downarrow}$. 
But each element of
$\{f_{1},\ldots, f_{m}\}^{\downarrow}$
is below an element such as $f_{i}$ and so $f \leq \theta$.
In addition, a compatible join of partial homeomorphisms is itself a partial homeomorphism.
We have proved that $A \subseteq  \theta^{\downarrow} \cap  \mathsf{I}^{\tiny cl}(X)$.
To prove the reverse inclusion, let $f$ be any surjective partial homeomorphism between the clopens of $X$ where $f \subseteq \theta$.
Then $f$ is the compatible join of partial homemorphisms each of the form $f \cap f_{i}$.
Thus $f \in A$.
We have therefore demonstrated the construction.

%%%%%%%%%%%%%%%%%%%%%%%%%%%%%%%%%%%%%%%%%%%%%%%%%%%%%%%%%%%%%%%%%%%%%%%%%%%%%%%%%%%%%%%%%%%%%%%%%%%%%%%%%%%%%%%%%%%%%%%%%%%%%%%%%
\subsection{The classical Thompson groups $G_{n,1}$} 

This section bears an analogous relationship to \cite{Birget2009},
as \cite{Birget2004} does to my original paper \cite{Lawson2007b}.\footnote{I am grateful to Richard Garner for reminding me of this.}
We now apply the constructions of this paper to the Boolean inverse monoids that arise
in constructing the Thompson groups $G_{n,1}$.
These Boolean inverse monoids were first described in \cite{Lawson2007b}
and then in \cite{LS1} with a more general perspective provided by \cite{Lawson2021b}.
According to \cite{Scott}, the starting point for constructing the classical Thompson groups  
involves free actions \cite{KKM}\footnote{My thanks to Victoria Gould for supplying this reference.} of free monoids $A_{n}^{\ast}$. 
We shall simplify things by considering only those Boolean inverse monoids that arise in the construction of the groups $G_{n,1}$.
This involves free monoids alone.

We refer the reader to \cite{BP} for more on strings.\footnote{`Strings' is the term common in semigroup theory along with `words'. You might know these objects as `lists'.}
Let $A_{n}$ be an $n$-element alphabet.
We shall usually denote its elements by $A_{n} = \{a_{1}, \ldots, a_{n}\}$,
although in the case $n = 2$ we shall often write $A_{2} = \{a,b\}$.
We shall always assume that $n \geq 2$.
The free monoid on $A_{n}$ is denoted by $A_{n}^{\ast}$ with identity $\varepsilon$, the empty string,
and composition concatenation.
If $x,y \in A_{n}^{\ast}$ then we say that $x$ and $y$ are {\em (prefix) incomparable} if $xA_{n}^{\ast} \cap yA_{n}^{\ast} = \varnothing$,
otherwise they are said to be {\em (prefix) comparable}.
A {\em finite}  subset $X$ of $A_{n}^{\ast}$ is said to be a {\em prefix code}
if the elements are pairwise prefix incomparable.
A prefix code $X$ is said to be a {\em maximal prefix code} if every 
element of $A_{n}^{\ast}$ is comparable with an element of $X$.
The smallest maximal prefix code is $\{\varepsilon\}$, which we call the {\em trivial maximal prefix code}.
Observe that $A_{n}$ itself is a maximal prefix code
which we call  a {\em caret}.
The set of all strings over $A_{n}$ of length $l$
also forms a maximal prefix code.
We shall say that they are {\em uniform of length $l$} denoted by $A_{n}^{l}$.

Denote by $A_{n}^{\omega}$ the set of all right-infinite strings over the $n$-element alphabet $A_{n}$.
This set is equipped with the topology in which the open sets of $X$ are the subsets $XA_{n}^{\omega}$ where $X \subseteq A_{n}^{\ast}$.
With this topology $A_{n}^{\omega}$ is the {\em Cantor space} and is a Boolean space.
The clopen subsets are precisely those where $X$ is finite.

We shall now define an inverse subsemigroup of $\mathsf{I}^{\tiny cl}(A_{n}^{\omega})$.
If $x \in A_{n}^{\ast}$ then the function $\lambda_{x} \colon A_{n}^{\omega} \rightarrow xA_{n}^{\omega}$ 
is defined by $w \mapsto xw$ where $w \in A_{n}^{\omega}$.
Observe that $\lambda_{\varepsilon}$ is the identity function on $A_{n}^{\omega}$,
and if $x = x_{1} \ldots x_{n}$ where $x_{i} \in A_{n}$ then
$\lambda_{x} = \lambda_{x_{1}} \ldots \lambda_{x_{n}}$.
Each $\lambda_{x}$ is a partial homeomorphism of the Cantor space.
We denote by $P_{n}$ the inverse submonoid of $\mathsf{I}^{\tiny cl}(A_{n}^{\omega})$
generated by the maps $\lambda_{x}$.
This is called the {\em polycyclic monoid on $n$ letters}.
Observe that the elements of $P_{n}$ are always partial homeomorphisms
between clopen subsets of the Cantor space.\\

\noindent
{\bf Notation} We shall denote the inverse of $\lambda_{x}$ by $\lambda_{x^{-1}}$.
We shall write $yx^{-1}$ instead of $\lambda_{yx^{-1}}$. \\

A typical non-zero element of $P_{n}$ has the form  $\lambda_{yx^{-1}}$; we shall call such an element {\em basic}. 
This maps $xw$ to $yw$, where $w$ is a right-infinite string.
Inside $P_{n}$,
the product assumes the following form
$$
(yx^{-1})(vu^{-1}) = \begin{cases}
0 & \text { If } x \text{ and } y \text{ are prefix incomparable}\\
yzu^{-1} &  \text{ if } v = xz\\
y(uz)^{-1} & \text{ if } x = vz.
\end{cases}
$$The nonzero idempotents of $P_{n}$ are the elements of the form $xx^{-1}$.
The natural partial order is given by $yx^{-1} \leq vu^{-1}$ iff
$(y,x) = (v,u)p$ for some finite string $p$.
You can find out more about the polycyclic inverse monoid in \cite{Lawson1998}.
We can translate definitions from the free monoid into the polycyclic inverse monoid. The following was proved as \cite[Corollary 3.4]{Lawson2007a}.

\begin{lemma}\label{lem:pmr} The set $\{x_{1}, \ldots, x_{n}\}$ is a prefix code in $A_{n}^{\ast}$
iff $\{x_{1}x_{1}^{-1}, \ldots, x_{n}x_{n}^{-1} \}$ is an orthogonal set in $P_{n}$.
\end{lemma}

Although the inverse monoid $P_{n}$ is important, it is not the inverse monoid we want.
We define this now.
Put $C_{n} = P_{n}^{\vee}$, within $\mathsf{I}^{\tiny cl}(A_{n}^{\omega})$.
This is a Boolean inverse monoid that we call the {\em Cuntz inverse monoid (of degree n)}.
The group of units of this Boolean inverse monoid is, in fact, the Thompson group $G_{n,1}$,
(although this will play no role in what we do).
See \cite{Lawson2007b} and also \cite{LS1}. 
The idempotents of $C_{n}$ are the identity functions defined on the clopen subsets of the Cantor space.
In a Boolean inverse monoid, every finite join is equal to a finite orthogonal join
by Lemma~\ref{lem:left-orthogonal} and its dual.
We now apply Lemma~\ref{lem:pmr}.
Let $X = \{x_{1}, \ldots, x_{r}\}$ and $Y = \{y_{1}, \ldots, y_{r}\}$
be prefix codes with the same number of elements both regarded as being ordered
as the elements are listed.
Define a function $f^{X}_{Y} \colon XA_{n}^{\omega} \rightarrow YA_{n}^{\omega}$
by $f^{X}_{Y}(x_{i}w) = y_{i}w$ where $w \in A_{n}^{\omega}$.
Every element of $C_{n}$ can be represented in this form.

We now describe the companion of $C_{n}$.
We shall work inside $\mathsf{S}(A_{n}^{\omega})$.
Define $H_{n}$ to be the union of all finite left-compatible sets of basic functions.
Clearly, $H_{n}$ is a right restriction monoid
whose set of projections forms a Boolean algebra.
Furthermore, $C_{n} \subseteq H_{n}$.
We can use Lemma~\ref{lem:left-orthogonal} to obtain a representation
of the elements of $H_{n}$.
Let $X$ be a prefix code and let $Y$ be any finite set (and so not a prefix code in general) with the same cardinality as $X$,  
where $X = \{x_{1},\ldots, x_{p}\}$ and $Y = \{y_{1}, \ldots, y_{p}\}$.
Define a function $f^{X}_{Y} \colon XA_{n}^{\omega} \rightarrow YA_{n}^{\omega}$
by $f^{X}_{Y}(x_{i}w) = y_{i}w$ where $w \in A_{n}^{\omega}$.
Every element of $H_{n}$ can be represented in this form.
The proof of (1) below is immediate from the definition
and the proof of (2) below follows from (1).

\begin{lemma}\label{lem:semigroup-three}\mbox{} 
\begin{enumerate}
\item The semigroup  $H_{n}$  is closed under binary left-compatible joins.

\item A finite left-compatible join in $H_{n}$ is closed under right multiplication by elements of $H_{n}$.

\end{enumerate}
\end{lemma}

The above lemma shows that $H_{n}$ is a Boolean right restriction monoid.
It is \'etale by construction.

\begin{lemma}\label{lem:semigroup-four} 
$\mathsf{Inv}(H_{n}) = C_{n}$.
\end{lemma}
\begin{proof} From the definition, partial units in $H_{n}$ must be injective functions on their domains of definition.
Let $f^{X}_{Y}$ be a partial unit in $H_{n}$.
We show that $Y$ is also a prefix code.
Suppose not. 
Then there exists $y_{i}$ and $y_{j}$ which are prefix comparable where $i \neq j$.
Without loss of generality, we may
suppose that $y_{i} = y_{j}p$ for some non-empty finite string $p$.
Then for any right-infinite string $w$, we have that $x_{j}pw$ and $x_{i}w$ are mapped to the same element $y_{j}pw$
by $f^{X}_{Y}$ and so this function is not injective.
\end{proof}

We now summarize what we have proved
in Lemma~\ref{lem:semigroup-three} and Lemma~\ref{lem:semigroup-four}.

\begin{theorem}\label{thm:h2}
The monoid $H_{n}$ is an \'etale Boolean right restriction monoid
with monoid of partial units isomorphic to $C_{n}$.
In particular, $H_{n}$ is the companion of $C_{n}$. 
\end{theorem}

%%%%%%%%%%%%%%%%%%%%%%%%%%%%%%%%%%%%%%%%%%%%%%%%%%%%%%%%%%%%%%%%%%%%%%%%%%%%%%%%%%%%%%%%%%%%%%%%%%%%%%%%%%%%%%%%%%%%%%%%%%%%%%%%%%
Before I plunge into the next remark, I need to say something about the structure of maximal prefix codes in $A_{n}^{\ast}$.
Everything I say is well-known, although I also give references to where proofs can be found. 
To determine when a prefix code $X$ is maximal, the following result is useful.
For a proof, see \cite[Proposition~1.2]{Lawson2021b}.

\begin{lemma}\label{lem:det-when-mpc} Working over the alphabet $A_{n}$,
the prefix code $X$ is maximal iff $XA_{n}^{\omega} = A_{n}^{\omega}$.
\end{lemma}

Let $X$ be a maximal prefix code.
We define two operations on $X$.
\begin{enumerate}

\item Choose any $x \in X$.
Then $X^{+} = X \setminus \{x\} \cup xA_{n}$ is also a maximal prefix code,
which we say is obtained from $X$ by means of {\em expansion by a caret} \cite[Lemma 5.4]{LS1}.

\item Suppose that there is a string $x$  such that $xA_{n} \subseteq X$.
(Such a string $x$ always exists if $X$ is not trivial.)
Then $X^{-} = (X \setminus xA_{n}) \cup \{x\}$ is a maximal prefix code,
which we say is obtained from $X$ by means of {\em reduction by a caret} \cite[Lemma 5.4]{LS1}.
\end{enumerate}

\begin{example}\label{ex:leo}
{\em Let $n = 2$ and $A_{2} = \{a,b\}$.
Then $X = \{a,ba,bb\}$ is a maximal prefix code.
In this case, $X^{-} = \{a,b\}$,
where we chose $x = b$,
and $X^{+} = \{aa,ab,ba,bb\}$, where we chose $x = a$.}
\end{example}

Every maximal prefix code is either the trivial maximal prefix code, $\{\varepsilon\}$,
or is obtained from the trivial maximal prefix code by a finite sequence of caret expansions \cite[Theorem 3.3]{Lawson2021b}.

\begin{example}\label{ex:taurus}
{\em We can construct the maximal prefix code $X$ of Example~\ref{ex:leo}
by using the following sequence of caret expansions:
$\{\varepsilon \} \rightarrow \{a,b\} \rightarrow \{a,ba,bb\}$. 
}
\end{example}

Every maximal prefix code may be expanded to one which is uniform
by inserting suitable carets.
We saw this in Example~\ref{ex:leo}.
We can deduce from this that if $X$ and $X'$ are maximal prefix codes over the same alphabet
then $X'$ can be obtained from $X$ by means of some finite sequence of caret expansions and caret reductions.
To see why, choose an $l \geq 0$ equal to or larger than the longest string in $X \cup X'$.
Then both $X$ and $X'$ lead to a uniform maximal prefix code $W$ where the length of every string is $l$.
We can pass from $X$ to $W$ by means of a sequence of caret expansions and from $W$ to $X'$ by a sequence of caret reductions. 

Using the idea of caret expansion, 
we can show that there is a close connection between maximal prefix codes over $A_{n}$ and expressions involving variables and operators $\lambda$ of arity $n$.
For our later purposes, $\lambda$ will be written to the right of its arguments. 

\begin{lemma}\label{lem:max-pc}
Every non-trivial maximal prefix code over $A_{n}$ arises from 
an expression involving variables and an operator $\lambda$ of arity $n$.
\end{lemma}
\begin{proof}
Let $t$ be such an expression.
Draw the rooted tree of $t$ where the leaves are labelled by the variables
and all other vertices are labelled by $\lambda$.
Erase all leaves and their edges.
In this way, we obtain a tree with all vertices labelled by $\lambda$
in which, apart from the leaves, every vertex has $n$-descendants.
We can now construct a maximal prefix code. 
Label the edges descending from a vertex by $a_{1}, \ldots, a_{n}$, respectively. 
By composing along adjacent sequences of edges, we obatain a set of strings
which forms a maximal prefix code.
That we obtain all maximal prefix codes in this way
follows from the fact that every non-trivial maximal prefix code arises
from the trivial maximal prefix code by caret expansions.
\end{proof}

\begin{example}{\em The maximal prefix code corresponding to
$$t = ((x,y)\lambda, ((u,v)\lambda, (s,t)\lambda)\lambda)\lambda$$ 
is
$\{a,ba,bb\}$.
}
\end{example}

%%%%%%%%%%%%%%%%%%%%%%%%%%%%%%%%%%%%%%%%%%%%%%%%%%%%%%%%%%%%%%%%%%%%%%%%%%%%%%%%%%%%%%%%%%%%%%%%%%%%%%%%%%%%%%%%%%%%%%%%%%%%%%%%%%%%%
\begin{remark}{\em We now make the connection with our paper \cite{Lawson2021b} (but, with the benefit of hindsight, our presentation here is a little sharper)
and the important papers \cite{G1, G2} which locate part of operator theory
within classical group theory.
We begin with the \'etale Boolean right restriction monoid $H_{n}$.
The elements of $H_{n}$ which belong to $\mathsf{Tot}(H_{n})$ have the form
$f^{X}_{Y}$ where $X$ is a maximal prefix code by Lemma~\ref{lem:det-when-mpc}.
We now define a class of structures from universal algebra,
which, at first sight, seems to have nothing to do with \'etale Boolean right restriction monoids.

By a {\em Cantor algebra} we mean a structure of universal algebra $(X,\alpha_{1}, \ldots, \alpha_{n}, \lambda)$
where $\alpha_{1}, \ldots, \alpha_{n}$ are unary operations and $\lambda$ is an $n$-operation 
(it is convenient to write all operations on the right of their arguments) 
satisfying the two laws below where $x,x_{1}, \ldots, x_{n} \in X$:
\begin{description}
\item[CA1] $(x\alpha_{1}, \ldots, x\alpha_{n}) \lambda = x$
\item[CA2] $(x_{1}, \ldots, x_{n})\lambda \alpha_{i} = x_{i}$. 
\end{description}
If $u \in \{1, \ldots, n\}^{\ast}$ is a non-empty string such that $u = x_{1} \ldots x_{m}$
where $x_{j} \in \{1, \ldots, n\}$ then we write $\alpha_{u}$ instead of $\alpha_{x_{1}} \ldots \alpha_{x_{n}}$.
As with any entity from universal algebra, we can choose a set of variables
and then define the set of terms over those variables in the Cantor algebra.
I shall denote by $\mathcal{C}[x]$ all such terms in the variable $x$.
For example, if $n = 2$, then an example of a  term in $\mathcal{C}[x]$  is $((x\alpha_{1}, x\alpha_{2})\lambda \alpha_{2}, (x\alpha_{2}\alpha_{1}, x\alpha_{1})\lambda ) \lambda$. 
However, because of the law (CA2), we can represent every element of  $\mathcal{C}[x]$ 
by means of an  $n$-ary tree whose leaves are labelled by elements of the form $x\alpha_{u}$
and whose vertices are labelled by $\lambda$.

The connection between Cantor algebras and the particular \'etale Boolean right restriction monoid $H_{n}$ now follows.
Let $A_{n} = \{a_{1}, \ldots, a_{n}\}$.
If $f \in \mathsf{Tot}(H_{n})$ define 
$$f\alpha_{i} = f \lambda_{a_{i}}\footnote{There is a typo on page 209 of \cite{Lawson2021b}.}.$$  
If $f_{1}, \ldots, f_{n} \in \mathsf{Tot}(H_{n})$ define 
$$(f_{1}, \ldots, f_{n})\lambda = f_{1}\lambda_{a_{1}}^{-1} \cup \ldots \cup f_{n}\lambda_{a_{n}}^{-1}.$$
These operations are well-defined on $\mathsf{Tot}(H_{n})$. 
Law (CA1) holds because
$$a_{1}a_{1}^{-1} \vee \ldots \vee a_{n}a_{n}^{-1} = 1,$$
and law (CA2) holds because
$a_{j}a_{j}^{-1} \perp a_{i}a_{i}^{-1}$ if $i \neq j$ and $a_{i}^{-1}a_{i} = 1$.
Thus we obtain a Cantor algebra because $H_{n}$ contains a copy of $P_{n}$.
If $x = a_{i}$ then it will be convenient to put $\alpha_{x} = \alpha_{i}$.
More generally, if $x = x_{1} \ldots x_{m}$ where $x_{i} \in A_{n}$ then define
$f \alpha_{x} = f\alpha_{x_{1}} \ldots \alpha_{x_{m}}$.
We prove four things about this Cantor algebra,
the details of which can be found in \cite{Lawson2021b}.
The central idea is based on Lemma~\ref{lem:max-pc}
and the fact that in Cantor algebras, the law (CA2) implies that
the terms of a Cantor algebra are always equal to terms of a particular form.

1. {\em The Cantor algebra, $\mathsf{Tot}(H_{n})$ is generated by $1 = \varepsilon \varepsilon^{-1}$.}
(In what follows, when we write a string $x$ we always mean $\lambda_{x}$.)  
We begin by defining special elements of $\mathsf{Tot}(H_{n})$ called {\em $\lambda$-expressions}:
\begin{enumerate}
\item If $y_{1}, \ldots, y_{n}$ are finite strings (recall my point above) then $(y_{1}, \ldots, y_{n})\lambda$ is a $\lambda$-expression;
this is just the element $\bigvee_{i=1}^{n} y_{i}a_{i}^{-1}$ of $\mathsf{Tot}(H_{n})$.
\item If $f_{1}, \ldots, f_{n}$ are $\lambda$-expressions, so too is $(f_{1}, \ldots, f_{n})\lambda$. 
\item Every $\lambda$-expression is obtained in a finite number of steps using rules (1) and (2) above.
\end{enumerate} 
Observe that every $\lambda$-expression is an element of $\mathsf{Tot}(H_{n})$.
An element $f$ of $\mathsf{Tot}(H_{n})$ can be written $f = f^{X}_{Y}$ where $X$ is a maximal prefix code.
In what follows, we shall assume that $X$ is a non-trivial maximal prefix code. 
There is therefore a finite string $x$ such that $xA_{n} \subseteq X$ by the proof of  \cite[Theorem 3.4]{Lawson2021b}.
Put $X^{-} = X \setminus xA_{n} \cup \{x\}$, a smaller maximal prefix code.
Then $X = (X^{-} \setminus \{x\}) \cup xA_{n}$.
By assumption, $xa_{1}, \ldots, xa_{n} \in X$.
If the elements of $Y$ corresponding to these elements are $y_{1}', \ldots, y_{n}'$, respectively,
then $\bigvee_{i=1}^{n} y_{i}'(xa_{i})^{-1} = \left( \bigvee_{i=1}^{n} y_{i}'a_{i}^{-1} \right)x^{-1}$. 
We shall abbreviate $\bigvee_{i=1}^{n} y_{i}'a_{i}^{-1}$ by $(y_{1}', \ldots, y_{n}')\lambda$.
We may therefore write $f = \bigvee_{x' \in X^{-}} f_{x'}(x')^{-1}$,
where each $f_{x'}$ is either an element of the free monoid or a $\lambda$-expression.  
Since $X^{-}$ is a smaller maximal prefix code, this process may continue.
In this way, every element of $\mathsf{Tot}(H_{n})$ of the form $f^{X}_{Y}$, where $X$ is not the trivial maximal prefix code, may be written as a $\lambda$-expression.
We now need to deal with elements of the form $y\varepsilon^{-1}$.
If $y$ is non-empty, then this function can be written $\alpha_{y}$.
If $y$ is the empty string, then $\varepsilon \varepsilon^{-1} = 1$, the identity function.
What results I shall call a {\em ($\lambda,\alpha$)-expression}.
What we have proved is the following.
Every term in $\mathcal{C}[x]$ can be interpreted as a specific term in the Cantor algebra $\mathsf{Tot}(H_{n})$.
Replace the $x$ by $1$.
This produces an element of $\mathsf{Tot}(H_{n})$ and every element of $\mathsf{Tot}(H_{n})$
can be written in this way.
Thus every element of $\mathsf{Tot}(H_{n})$ can be written as a 
($\lambda,\alpha$)-expression.

2. {\em The structure
$$(\mathsf{Tot}(H_{n}), \alpha_{1}, \ldots, \alpha_{n}, \lambda)$$
is the free Cantor algebra on one generator.}
From universal algebra, we know that the free Cantor algebra on one generator $x$ exists.
We denote it by $\mathscr{C}$.
This admits a surjective (by what we proved above) homomorphism onto  
$$(\mathsf{Tot}(H_{n}), \alpha_{1}, \ldots, \alpha_{n}, \lambda)$$
when we map the single generator of $\mathscr{C}$ to $1$.
It remains to show that this homomorphism is injective.
A typical element of $\mathscr{C}$ can be represented by means of an $n$-ary tree (as above).
The homomorphism can then be described by its effect on that tree.
The $x$ is mapped to $1$ and the abstract operations of $\mathscr{C}$ are mapped
to the concrete Cantor algebra operations of $\mathsf{Tot}(H_{n})$. 
In this way, an element of $\mathscr{C}$ is mapped to an element of $\mathsf{Tot}(H_{n})$
represented as a ($\lambda,\alpha$)-expression.
This means that it can be presented in the form $f^{X}_{Y}$, where the maximal prefix code $X$
is obtained directly from the pattern of $\lambda$s in the ($\lambda,\alpha$)-expression.
If one element of $\mathscr{C}$ is mapped to $f^{X}_{Y}$ 
and another element of $\mathscr{C}$ is mapped to $f^{X'}_{Y'}$,
we now have to investigate what happens when $f^{X}_{Y} = f^{X'}_{Y'}$.
It is here that we again use the theory of maximal prefix codes.
If $f^{X}_{Y} = f^{X'}_{Y'}$, where $X$ and $X'$ are maximal prefix codes,
then $X'$ can be obtained from $X$ by a sequence of caret expansions and reductions.
It is enough to consider what happens when we apply a caret expansion.
Suppose that $f = f^{X}_{Y}$ where $X = \{x_{1}, \ldots, x_{m}\}$ and $Y = \{y_{1}, \ldots, y_{m}\}$.
Then $f = \bigvee_{i=1}^{m}y_{i}x_{i}^{-1}$.
Relabelling, if necessary, we may assume that $x_{1}$ is the element to which caret expansion is applied.
Define $Y'$ to consist of the sequence $y_{1}a_{1}, \ldots, y_{1}a_{n}$ followed by $y_{2}, \ldots, y_{m}$.
Then 
$$f = f^{X^{+}}_{Y'} = y_{1}a_{1}(x_{1}a_{1})^{-1} \vee \ldots \vee y_{1}a_{n}(x_{1}a_{n}^{-1}) \vee \left( \bigvee_{i=2}^{m} y_{i}x_{i}^{-1} \right).$$
It follows that $f = (y_{1}\alpha_{1}, \ldots, y_{1}\alpha_{n})\lambda x_{1}^{-1} \vee \left( \bigvee_{i=2}^{m} y_{i}x_{i}^{-1} \right)$.
We have replaced $y_{1}$ by $(y_{1}\alpha_{1}, \ldots, y_{1}\alpha_{n})\lambda$.
This is an application of the law (CA1).
It follows that if two elements of $\mathscr{C}$ are mapped to the same element of $\mathsf{Tot}(H_{n})$
then one element can be converted into the other by applying the laws (CA1) and (CA2).
This means the two elements are equal in $\mathscr{C}$.
We have thus proved that 
$$(\mathsf{Tot}(H_{n}), \alpha_{1}, \ldots, \alpha_{n}, \lambda)$$
is the free Cantor algebra on one generator.
It is, therefore, in some sense, canonical.

3. {\em We now describe the endomorphisms of this structure.}
If $f \in \mathsf{Tot}(H_{n})$ then we may assume that $f = f^{X}_{Y}$
where $X$ is uniform of length $l$.
We may write
$X = X_{1}a_{1} \cup \ldots \cup X_{1}a_{n}$
where $X_{1}$ is uniform of length $l-1$.
Partition $Y$ into sets $Y_{1},\ldots, Y_{n}$
so that $X_{1}a_{i}$ corresponds to $Y_{i}$.
We may then write $f^{X}_{Y} = \bigvee_{i=1}^{n} f^{X_{1}}_{Y_{i}}a_{i}^{-1}$.
Thus $f^{X}_{Y} = (f^{X_{1}}_{Y_{1}}, \ldots, f^{X_{1}}_{Y_{n}})\lambda$.
If $\theta$ is an endomorphism of the Cantor algebra then
$\theta (f^{X}_{Y}) = (\theta (f^{X_{1}}_{Y_{1}}), \ldots, \theta (f^{X_{1}}_{Y_{n}}))\lambda$.
By induction, we therefore have to compute $\theta (y\varepsilon^{-1})$.
We need to calculate $\theta (ay)$ where $a \in A_{n}$ and $y$ is any finite string.
Observe that $\lambda_{a} \lambda_{y} = (1)\alpha_{a} \alpha_{y}$.
It follows that $\theta (ay) = (\theta (1))\alpha_{a}\alpha_{y} = \theta (a)\alpha_{y}$.
Thus $\theta (ay) = \theta (a)y$.
We have proved that if  $\theta$ is an endomorphism then $\theta (f) = \bigvee \theta (y_{i})x_{i}^{-1}$.
In particular, $\theta (1) = \theta (a_{1})a_{1}^{-1} \vee \ldots \vee \theta (a_{n})a_{n}^{-1}$.
It is now routine to check that $\theta (f^{X}_{Y})$ is the same as $\theta (1)f^{X}_{Y}$.
This proves that the effect of the Cantor algebra endomorphism  $\theta$ is left multiplication by the element $\theta (1)$.

4. {\em If the endomorphism $\theta$ is actually an automorphism, then it can be shown that $\theta (1)$ is invertible.}
This now shows that the group of automorphisms of the Cantor algebra $(\mathsf{Tot}(H_{n}), \alpha_{1}, \ldots, \alpha_{n}, \lambda)$
is isomorphic to the group of units of $H_{n}$, which is the Thompson group $G_{n,1}$.
See \cite{Lawson2021b} for details.
This example explains why the fact that $\mathsf{Etale}(C_{n}) = H_{n}$
should be important in group theory.}
\end{remark}

%%%%%%%%%%%%%%%%%%%%%%%%%%%%%%%%%%%%%%%%%%%%%%%%%%%%%%%%%%%%%%%%%%%%%%%%%%%%%%%%%%%%%%%%%%%%%%%%%%%%%%%%%%%%%%%%%%%%%%%%%%%%%%%%%%%%%%%%%%%%
\begin{remark}{\em  Here is a broader description of what we have accomplished.
Let $P_{n}$ be the polycyclic monoid on $n$ generators \cite{Lawson1998}.
Then there is an injective homomorphism of semigroups
$P_{n} \rightarrow \mathsf{I}^{\tiny cl}(A_{n}^{\omega})$ whose isomorphic image consists of the basic functions and zero.
This is an example of what we called a {\em strong representation} of the polycyclic monoid \cite{Lawson2009}.
This was our original approach to constructing the Thompson groups.
See \cite{Lawson2021b} for a retrospective.
A parallel, but more general, approach was pioneered by \cite{Hughes}.
See, for example, \cite[Example 4.1]{Hughes}.
Observe that Hughes works from geometry whereas we work from language theory.
The connection with the theory of inverse semigroups is slightly obscured by the approach Hughes adopts,
but his \cite[Definition 3.1]{Hughes} is really the definition of a particular kind of inverse semigroup
of which the polycyclic inverse monoids are special cases.
Hughes is working with ultrametric spaces in which closed balls are also open
and if two balls intersect then one must be contained in the other.
This parallels what happens in free monoids in that if two finite strings 
are comparable then one must be the prefix of the other.}
\end{remark}

%%%%%%%%%%%%%%%%%%%%%%%%%%%%%%%%%%%%%%%%%%%%%%%%%%%%%%%%%%%%%%%%%%%%%%%%%%%%%%%%%%%%%%%%%%

\end{document}